\documentclass[a4paper,reqno, 11pt]{amsart}  
\usepackage[DIV=12, oneside]{typearea}
\usepackage[utf8]{inputenc}
\usepackage[T1]{fontenc}
\usepackage[english]{babel}
\usepackage[centertags]{amsmath}
\usepackage{amstext,amssymb,amsopn,amsthm}
\usepackage{mathrsfs}
\usepackage{dsfont}
\usepackage{bbm}
\usepackage{thmtools}
\usepackage{graphicx}
\usepackage[titletoc,title]{appendix}
\usepackage{etexcmds}

\usepackage[backgroundcolor=white, bordercolor=blue,
linecolor=blue]{todonotes}
\usepackage[colorlinks=true, linkcolor=black, citecolor=black]{hyperref}
\usepackage{enumitem}
\setlist[enumerate]{itemsep=0mm}

\addto\extrasenglish{}
\addto\extrasenglish{}
\addto\extrasenglish{}

\declaretheorem[name=Theorem, numberwithin=section]{theorem}

\newtheorem{lemma}[theorem]{Lemma}

\newtheorem{definition}[theorem]{Definition}

\newtheorem*{example*}{Example}

\declaretheoremstyle[bodyfont=\normalfont]{remark-style}
\declaretheorem[name={Remark}, style=remark-style, unnumbered]{remark}

\numberwithin{equation}{section}

\theoremstyle{plain}

\newcommand{\N}{\mathds{N}}
\newcommand{\R}{\mathds{R}}

\def\hmath$#1${\texorpdfstring{{\rmfamily\textit{#1}}}{#1}}

\newcommand{\cE}{\mathcal{E}}

\newcommand{\BIGOP}[1]
{
\mathop{\mathchoice%
{\raise-0.22em\hbox{\huge $#1$}}%
{\raise-0.05em\hbox{\Large $#1$}}{\hbox{\large $#1$}}{#1}}}

\def\Xint#1{\mathchoice
   {\XXint\displaystyle\textstyle{#1}}%
   {\XXint\textstyle\scriptstyle{#1}}%
   {\XXint\scriptstyle\scriptscriptstyle{#1}}%
   {\XXint\scriptscriptstyle\scriptscriptstyle{#1}}%
   \!\int}
\def\XXint#1#2#3{{\setbox0=\hbox{$#1{#2#3}{\int}$}
     \vcenter{\hbox{$#2#3$}}\kern-.5\wd0}}

\def\dashint{\Xint-}
\newcommand{\BIGboxplus}{\mathop{\mathchoice%
{\raise-0.35em\hbox{\huge $\boxplus$}}%
{\raise-0.15em\hbox{\Large $\boxplus$}}{\hbox{\large $\boxplus$}}{\boxplus}}}

\DeclareMathOperator{\dist}{dist}

\DeclareMathOperator{\supp}{supp}

\renewcommand{\d}{\textnormal{d}}

\newcommand{\1}{{\mathbbm{1}}}

\begin{document}
\allowdisplaybreaks
 \title{Upper heat kernel estimates for nonlocal operators via Aronson's method}

\allowdisplaybreaks

\author{Moritz Kassmann}
\author{Marvin Weidner}

\address{Fakult\"{a}t f\"{u}r Mathematik\\Universit\"{a}t Bielefeld\\Postfach 100131\\D-33501 Bielefeld}
\email{moritz.kassmann@uni-bielefeld.de}
\urladdr{www.math.uni-bielefeld.de/$\sim$kassmann}

\address{Fakult\"{a}t f\"{u}r Mathematik\\Universit\"{a}t Bielefeld\\Postfach 100131\\D-33501 Bielefeld}
\email{mweidner@math.uni-bielefeld.de}


\thanks{Moritz Kassmann and Marvin Weidner gratefully acknowledge financial support by the German Research Foundation (SFB 1283 - 317210226 resp. GRK 2235 - 282638148).}


\allowdisplaybreaks

\begin{abstract}
In his celebrated article, Aronson established Gaussian bounds for the fundamental solution to the Cauchy problem governed by a second order divergence form operator with uniformly elliptic coefficients. We extend Aronson's proof of upper heat kernel estimates to nonlocal operators whose jumping kernel satisfies a pointwise upper bound and whose energy form is coercive. A detailed proof is given in the Euclidean space and extensions to doubling metric measure spaces are discussed.  
\end{abstract}

\allowdisplaybreaks

\maketitle

\section{Introduction}

\subsection{Background}
Let $a_{i,j} : [0,T] \times \R^d \to (0,\infty)$, $i,j = 1,\dots,d$, be bounded, measurable coefficients which satisfy the usual uniform ellipticity condition. A celebrated result by D.G. Aronson from 1967 says that the fundamental solution $\Gamma(y,s;x,\eta)$ to the Cauchy problem
\begin{equation}
\label{eq:localCP}
\begin{cases}
\partial_t u - \partial_i(a_{i,j}\partial_j u) &= 0,~~ \text{ in } (\eta,T) \times \R^d, \\
u(\eta) &= u_0
\end{cases}
\end{equation}
satisfies the following two-sided estimate for all $0 \le \eta < s < T$, and $x,y \in \R^d$:
\begin{equation}
\label{eq:localHKB}
c_1 t^{-\frac{d}{2}}e^{-c_2 \frac{\vert x -y \vert^2}{s - \eta}} \le \Gamma(y,s;x,\eta) \le c_3 t^{-\frac{d}{2}}e^{-c_4 \frac{\vert x -y \vert^2}{s - \eta}},
\end{equation}
where $c_1,c_2,c_3,c_4 > 0$ depend only on $d$ and the ellipticity constants. In other words, the fundamental solution of the classical heat equation $\partial_t u - \Delta u = 0$ is an upper and lower bound of $\Gamma$ up to multiplicative constants, see \cite{Aro67}, \cite{Aro68}. In this sense, the bounds \eqref{eq:localHKB} are robust in the class of second order divergence form operators with bounded, measurable, uniformly elliptic coefficients. 

Aronson's proof is closely related to the so-called DeGiorgi-Nash-Moser theory for parabolic differential operators of second order with bounded, measurable and uniformly elliptic coefficients. The proof heavily relies on H\"older regularity estimates and the parabolic Harnack inequality for solutions to \eqref{eq:localCP}, see also \cite{ArSe67}.\\
\cite{Aro68} has initiated several research studies on estimates for fundamental solutions to parabolic equations in various contexts. An important feature of this research is that it connects partial differential equations with geometry. This is due to the sensitivity of the heat kernel to the geometric properties of the underlying space. This phenomenon becomes apparent in the celebrated works \cite{CLY81}, \cite{LiYa86}, where the method of Aronson was generalized to prove heat kernel estimates on complete Riemannian manifolds with nonnegative Ricci curvature. Some of their arguments have been further refined and generalized in \cite{Dav89}, where an integral estimate for the heat kernel was established that is useful for proving the upper bound in \eqref{eq:localHKB}. We refer the interested reader to \cite{PoEi84}, \cite{Gri09}, \cite{Li12} and the references therein for more detailed expositions on this topic.

\subsection{Main results}
The goal of this article is to extend Aronson's proof of upper heat kernel estimates to integro-differential operators of the form
\[ L_t u (x) = \text{ p.v.} \int_{\R^d} (u(y)-u(x))k(t;x,y) \d y, ~~ t \in (0,T),~ x \in \R^d.\]
Such operators are determined by a jumping kernel $k : (0,T) \times \R^d \times \R^d \to [0,\infty]$ which is assumed to be symmetric, i.e., $k(t;x,y) = k(t;y,x)$ and satisfies a pointwise upper bound
\begin{equation}
\label{eq:kupper}\tag{$k_{\le}$}
 k(t;x,y) \le \Lambda \vert x-y \vert^{-d-\alpha}, ~~ t \in (0,T),~ x,y \in \R^d,
\end{equation}
for some given constant $\Lambda > 0$, and $\alpha \in (0,2)$, $0 < T \le \infty$.\\
Moreover, we assume that there is $\lambda > 0$ such that for any ball $B \subset \R^d$ and every $v \in H^{\alpha/2}(B)$:
\begin{equation}
\label{eq:klower}\tag{$\cE_{\ge}$}
\int_{B}\int_{B} (v(x) - v(y))^2 k(t;x,y) \d y \d x \ge \lambda [v]^2_{H^{\alpha/2}(B)} , ~~ t \in (0,T).
\end{equation}
\eqref{eq:klower} can be thought of as a coercivity assumption on $k$ and is substantially weaker than a pointwise lower bound. We refer the reader to \autoref{sec:prelim} for a more detailed discussion and to \autoref{sec:extensions} where we explain a possible extension of our method and replace \eqref{eq:klower} by a Faber-Krahn inequality.\\
We are now ready to state the main result of this article in the aforementioned setup. For a possible extension to doubling metric measure spaces and jumping kernels of mixed type, we refer to \autoref{thm:mainmixed}. 

\begin{theorem}
\label{thm:main}
Let $k : (0,T) \times \R^d \times \R^d \to [0,\infty]$ be symmetric and assume \eqref{eq:kupper}, \eqref{eq:klower}. Let $p(y,s;x,\eta)$ be the fundamental solution to the Cauchy problem
\begin{align}
\label{eq:CP}
\begin{cases}
\partial_t u - L_t u &= 0, ~~ \text{ in } (\eta,T) \times \R^d,\\
u(\eta) &= u_0 \in L^2(\R^d),
\end{cases}
\end{align}
where $\eta \in [0,T)$. Then there exists a constant $c > 0$ depending on $d,\alpha,\lambda,\Lambda$ such that for every $0 \le \eta < s < T$, and $x,y \in \R^d$:
\begin{equation}
\label{eq:uhke}
p(y,s;x,\eta) \le c(s-\eta)^{-\frac{d}{\alpha}}\left(1 + \frac{\vert x-y \vert^{\alpha}}{s-\eta}\right)^{-\frac{d+\alpha}{\alpha}}.
\end{equation}
\end{theorem}

Estimate \eqref{eq:uhke} states that the fundamental solution to \eqref{eq:CP} possesses the same upper bound as the fundamental solution to the fractional heat equation $\partial_t u + (-\Delta)^{\alpha/2} u = 0$, see \cite{BlGe60}.

Different versions of \autoref{thm:main} already exist in the literature. Let us give a brief account on the history of heat kernel bounds for nonlocal operators. Two-sided estimates of the form \eqref{eq:uhke} have been established by \cite{BaLe02}, \cite{ChKu03} using a probabilistic approach. They assume that the jumping kernel is pointwise comparable to $\vert x-y \vert^{-d-\alpha}$, $\alpha \in (0,2)$, from above and below. Their analysis of the upper heat kernel estimate heavily relies on \cite{CKS87}, where Davies' method was extended to a more general setup, including jump processes. Building upon this, \cite{BGK09} derived \eqref{eq:uhke} assuming \eqref{eq:kupper} and a Nash inequality.

In a series of articles, see \cite{GHL14}, \cite{GHH17}, \cite{GHH18}, the analysis of heat kernel estimates was extended to metric measure spaces  with walk dimension greater than $2$. The authors were able to characterize upper heat kernel estimates, as well as two-sided heat kernel estimates in terms of equivalent conditions on the jumping kernels and the geometry of the underlying space. Their approach does not use the underlying stochastic process and is based on certain comparison inequalities of the corresponding heat semigroups relying on the parabolic maximum principle. Note that Davies' method was extended to jumping kernels with jump index $\alpha >2$ in \cite{MuSa19}, \cite{HuLi18}. The aforementioned results assume certain homogeneity of the doubling measure space and do not deal with mixed-type jumping kernels.

In \cite{ChKu08}, \cite{CKKW21}, \cite{CKW21} upper and two-sided heat kernel estimates were investigated on doubling metric measures spaces for jumping kernels of mixed type. We would like to draw the reader's attention to Theorem 1.15 in \cite{CKW21}. In the case $\alpha < 2$, it states that upper heat kernel estimates of the form \eqref{eq:uhke} are equivalent to a pointwise upper bound on the jumping kernel and a Faber-Krahn inequality, which can be understood as an implicit lower bound. \\ A major difference between our approach and \cite{CKW21} is that our method relies on purely analytic arguments, while \cite{CKW21} makes essential use of the corresponding stochastic process. In  \autoref{thm:mainmixed} we extend our approach to doubling metric measure spaces and jumping kernels of mixed type. Let us mention that we prove on-diagonal heat kernel estimates with the help of a parabolic $L^\infty - L^1$-estimate, see \autoref{lemma:mixedLinftyL2}. This rather straightforward approach allows us to avoid truncation methods and the usage of the iteration techniques of \cite{Kig04}. 

In contrast to our setup, all jumping kernels in the results discussed above are assumed to be time-homogeneous. Note that it would require substantial effort to extend methods based on stochastic processes to situations with time-dependent jumping kernels. Heat kernel estimates for time-inhomogeneous jumping kernels were established in \cite{MaMi13}, \cite{MaMi13b} where the authors assume pointwise upper and lower bounds on the jumping kernel. Note that assuming pointwise lower bounds is more restrictive than \eqref{eq:klower}. The focus of these works lies on the treatment of an additional divergence-free drift of first order.

There are further results on heat kernel estimates for nonlocal operators, which are related to \autoref{thm:main}. For example, sharp two-sided estimates for jump processes on $\R^d$ with upper scaling index not strictly less than $2$ are established in \cite{BKKL19}. In \cite{KKK21}, heat kernel estimates for a certain class of jump processes with singular jumping measures are proved.

\subsection{Strategy of proof}
A main insight of Aronson's proof for second order differential operators is the observation that solutions $u$ to the Cauchy problem \eqref{eq:localCP} satisfy the weighted $L^2$-estimate
\begin{equation}
\label{eq:AronsonIdea}
\sup_{t \in (\eta,s)} \int_{\R^d} H(t,x) u^2(t,x) \d x \le \int_{\R^d} H(\eta,x) u_0^2(x) \d x
\end{equation}
for $0 \le \eta < s < T$, whenever $H$ satisfies
\begin{equation}
\label{eq:localHass}
C \vert \nabla H^{1/2} \vert^2 \le -\partial_t H, ~~ \text{ in } (\eta,s) \times \R^d,
\end{equation}
for a given number $C > 0$ depending on the ellipticity constants. \eqref{eq:localHass} is closely related to the famous Li-Yau inequality:
\begin{equation}
\label{eq:localLiYau}
\vert \nabla \log w \vert^2 \le \frac{d}{2 t} + \partial_t \log w.
\end{equation}
In fact, a direct computation reveals that the Gauss-Weierstrass kernel $w(t,x) = t^{-\frac{d}{2}}e^{-\frac{\vert x \vert^2}{4t}}$ satisfies \eqref{eq:localLiYau} with equality. By a scaling argument, it becomes evident that \eqref{eq:localHass} holds true for $H(t,x) = (C[t])^{\frac{d}{2}} w(C[t],x-y) = \exp(-\frac{\vert x-y \vert^2}{4 C [t]})$, where $[t] := 2(s-\eta) - (t-\eta)$ and $y \in \R^d$ can be chosen arbitrarily.\\
This insight suggests that some qualitative information on the decay of solutions to \eqref{eq:localCP} is encoded in the weighted $L^2$-estimate \eqref{eq:AronsonIdea}. Indeed, by combining \eqref{eq:AronsonIdea} with a localized $L^{\infty}-L^2$-estimate, as it was proved by Moser (\cite{Mos64}, \cite{Mos67}, \cite{Mos71}), one can deduce a qualitative estimate in terms of $H$, on how quick solutions to the Cauchy problem increase inside a ball that lies outside the support of the initial data. From such estimate, it is not difficult to deduce
\begin{equation*}
\left(\int_{\R^d \setminus B_{\sigma}(y)} \Gamma^2(y,s;z,\eta) \d z\right)^{1/2} \le c (s-\eta)^{-\frac{d}{4}}e^{-\frac{\sigma^2}{32C (s-\eta)}}
\end{equation*}
for every $\sigma > 0$ and $0 \le \eta < s < T$ with $s-\eta \le \sigma^2$. Together with the on-diagonal estimate $\Gamma(y,s;x,\eta) \le c(s-\eta)^{-\frac{d}{2}}$, one deduces the upper bound in \eqref{eq:localHKB} via a standard argument.

We would like to point out that \eqref{eq:localLiYau} is at the core of the works \cite{LiYa86}, \cite{CLY81}, where \eqref{eq:localLiYau} was used to derive a parabolic Harnack inequality on Riemannian manifolds with nonnegative Ricci-curvature and to establish its equivalence to Gaussian heat kernel bounds.

Next, let us summarize how we adapt Aronson's proof to integro-differential operators. First, we require a nonlocal analog of \eqref{eq:AronsonIdea}. We prove that there exist functions $H$ that satisfy
\begin{equation}
\label{eq:nonlocalHass}
C \Gamma^{\alpha}_{\rho} (H^{1/2},H^{1/2}) \le -\partial_t H, ~~ \text{ in } (\eta,s) \times \R^d
\end{equation}
given $C,\rho > 0$. Here, $\Gamma^{\alpha}_{\rho}$ denotes the $\rho$-truncated carr\'e du champ operator of order $\alpha \in (0,2)$, which is defined as follows:
\begin{equation*}
\Gamma^{\alpha}_{\rho}(f,f)(x) = \int_{B_{\rho}(x)} (f(x)-f(y))^2 \vert x-y \vert^{-d-\alpha}\d y, ~~ x \in \R^d.
\end{equation*}
By a careful choice of a function $H$ satisfying \eqref{eq:nonlocalHass}, we to deduce a nonlocal analog of \eqref{eq:AronsonIdea} for solutions to the $\rho$-truncated Cauchy problem, see \autoref{lemma:Aronson}. The corresponding integro-differential operator only takes into account differences up to distance $\rho$. In order to prove an a priori bound for the corresponding fundamental solution $p_{\rho}$ (see \autoref{thm:offdiagtrunc}), we derive a parabolic $L^{\infty}-L^2$-estimate in the spirit of \cite{Str18}, see \autoref{lemma:truncL2}. The estimate involves a nonlocal, truncated tail-term which requires special treatment. 
In a final step, we obtain the desired upper heat kernel estimate \eqref{eq:uhke} for $p$  by gluing together short and long jumps. Such argument is by now standard in the theory of jump processes.

Last, we explain several difficulties that occur when avoiding the detour via the truncated jumping kernel. \\
First, the corresponding $L^{\infty}-L^2$-estimate involves a non-truncated tail-term which cannot be controlled without any further assumptions on $k$.\\
Second, finding suitable weight functions $H$ which satisfy a non-truncated version of \eqref{eq:nonlocalHass} is a challenging task in the light of the following observation: The corresponding nonlocal analog of the Li-Yau inequality, which would imply an estimate of the form \eqref{eq:AronsonIdea} for solutions to \eqref{eq:CP}, reads as follows:
\begin{equation}
\label{eq:nonlocalLiYau}
\frac{\Gamma^{\alpha}(w_{\alpha}^{1/2},w_{\alpha}^{1/2})}{w_{\alpha}} \le \frac{d}{\alpha t} + \partial_t \log(w_{\alpha}).
\end{equation}
However, one can show that the fundamental solution $w_{\alpha}(t,x)$ to $\partial_t u + (-\Delta)^{\alpha/2}u = 0$ does not satisfy \eqref{eq:nonlocalLiYau}.
Let us give a quick proof of this fact. First, note that $w_{\alpha}$ is a radial function and satisfies
\begin{equation*}
\frac{d}{\alpha t} w_{\alpha} + \partial_t w_{\alpha} = -\frac{\vert x \vert}{\alpha t} \partial_{\vert x \vert} w_{\alpha}.
\end{equation*}
For a proof of this identity, we refer to (2.5) in \cite{Vaz18}. Consequently,
\begin{equation*}
\frac{d}{\alpha t} + \partial_t \log w_{\alpha}(t,0) = 0,
\end{equation*}
\enlargethispage{3ex}but this is a contradiction to \eqref{eq:nonlocalLiYau} since $\Gamma^{\alpha}(w_{\alpha}^{1/2},w_{\alpha}^{1/2})(t,0) > 0$ for $t > 0$.

\subsection{Outline}
This article is separated into five sections. In \autoref{sec:prelim} we present several auxiliary results that we need in our proof. \autoref{sec:method} contains the derivation of the upper heat kernel bounds and proves \autoref{thm:main}. In \autoref{sec:extensions}, we explain how our method can be applied to jumping kernels of mixed type on metric measure spaces. In \autoref{sec:appendix}, we provide a proof of a gluing lemma which differs from \cite{GrHu08} due to the time-inhomogeneity of the jumping kernels under consideration.

\section{Preliminaries}
\label{sec:prelim}

In this section we provide several auxiliary results that will be required for the proof of \autoref{thm:main} in \autoref{sec:method}.

Let $k : (0,T) \times \R^d \times \R^d \to [0,\infty]$ be a symmetric jumping kernel satisfying the pointwise upper bound \eqref{eq:kupper} and the coercivity condition \eqref{eq:klower}.\\
Let us make a few comments on assumption \eqref{eq:klower}: First of all,  \eqref{eq:klower} can be regarded as a nonlocal substitute of the classical uniform ellipticity condition for local operators. In fact, it is considerably weaker than a pointwise lower bound on the jumping kernel, since \eqref{eq:klower} allows for jumping kernels that might degenerate in certain directions, as for example kernels that are supported on double cones. We refer the interested reader to \cite{ChSi20} for an investigation of such condition.\\
A coercivity assumption like \eqref{eq:klower} is crucial to our approach since it is needed for the $L^{\infty}-L^2$-estimate (\autoref{lemma:truncL2}) and also the on-diagonal heat kernel bound (\autoref{thm:ondiag}). In the literature, lower bounds on jumping kernels are often introduced through functional inequalities,  e.g. in \cite{BGK09}, or \cite{CKW21} where the authors assume a Nash -, or a Faber-Krahn inequality. We point out that such assumption would have been possible also in this work, since the proofs of \autoref{lemma:truncL2}, \autoref{thm:ondiag} can be changed accordingly (see \autoref{sec:extensions}). For a discussion on the equivalence of Nash - and Faber-Krahn inequalities, we refer the reader to \cite{CKW21}. 

For any $\rho > 0$, we define the truncated jumping kernel $k_{\rho}$ via 
\begin{equation*}
k_{\rho}(t;x,y) = k(t;x,y) \mathbbm{1}_{\{\vert x-y \vert \le \rho\}}(x,y).
\end{equation*}
The associated integro-differential operator $L_t^{\rho}$ is defined as
\[ L_t^{\rho} u (x) = \text{ p.v.} \int_{\R^d} (u(y)-u(x))k_{\rho}(t;x,y) \d y.\]

\begin{definition}
We say that a function $u \in L^2_{loc}((\eta,T);H^{\alpha/2}(\R^d))$ with $\partial_t u \in L^1_{loc}((\eta,T) ; L^2_{loc}(\R^d))$ solves the Cauchy problem associated with $k$ in $(\eta,T) \times \R^d$:
\begin{align}
\label{eq:CP2}
\begin{cases}
\partial_t u - L_t u &= 0, ~~ \text{ in } (\eta,T) \times \R^d,\\
u(\eta) &= u_0 \in L^2(\R^d),
\end{cases}
\end{align}
if for every $\phi \in H^{\alpha/2}(\R^d)$ with $\supp(\phi)$ compact, it holds
\begin{align}
\int_{\R^d} \partial_t u(t,x) \phi(x) \d x + \cE_t(u(t),\phi) &= 0, ~~ \text{ a.e. } t \in (\eta,T),\\
\Vert u(t) - u_0 \Vert_{L^2(\R^d)} &\to 0, ~~ \text{ as } t \searrow \eta,
\end{align}
where we write
\begin{equation}
\cE_t(u,v) = \int_{\R^d} \int_{\R^d} (u(t,x) - u(t,y))(v(t,x)-v(t,y))k(t;x,y) \d y \d x
\end{equation}
for the family of energy forms $(\cE_t)_{t \in (\eta,T)}$ associated with $k$.
\end{definition}

Solutions to the $\rho$-truncated Cauchy problem associated with $k$ in $(\eta,T) \times \R^d$
\begin{align}
\label{eq:truncCP}
\begin{cases}
\partial_t u - L_t^{\rho} u &= 0, ~~ \text{ in } (\eta,T) \times \R^d,\\
u(\eta) &= u_0 \in L^2(\R^d),
\end{cases}
\end{align}
are defined in an analogous way, replacing $k$ by $k_{\rho}$.

Throughout this article, we will assume that the fundamental solutions to the Cauchy problem \eqref{eq:CP2} and the $\rho$-truncated Cauchy problem \eqref{eq:truncCP} exist, i.e., that the solution $u$ to \eqref{eq:CP2} and $u_{\rho}$ to \eqref{eq:truncCP} are unique, and have the representation
\begin{equation}
\label{eq:sgrepresentation}
u(s,y) = \int_{\R^d} p(y,s;x,\eta) u_0(x) \d x,~~ u_{\rho}(s,y) = \int_{\R^d} p_{\rho}(y,s;x,\eta) u_0(x) \d x, ~~ s \in (\eta,T), ~ y \in \R^d,
\end{equation}
where $p, p_{\rho} : (0,T) \times \R^d \times [0,T) \times \R^d \to [0,\infty]$
satisfy the following properties for all $0 \le \eta < t < s < T$, $x,y \in \R^d$:
\begin{align}
\label{eq:psymm}
p(y,s;x,\eta) = p(x,s;y,\eta) > 0&,~~ p_{\rho}(y,s;x,\eta) = p_{\rho}(x,s;y,\eta) > 0,\\
\label{eq:pint}
\int_{\R^d} p(y,s;x,\eta) \d x = 1&, ~~ \int_{\R^d} p_{\rho}(y,s;x,\eta) \d x = 1,\\
\label{eq:psemi}
p(y,s;x,\eta) = \int_{\R^d}p(y,s;z,t)&p(z,t;x,\eta)\d z,~~ p_{\rho}(y,s;x,\eta) = \int_{\R^d}p_{\rho}(y,s;z,t)p_{\rho}(z,t;x,\eta)\d z.
\end{align}
In the following, we will denote the unique solutions to \eqref{eq:CP2} and \eqref{eq:truncCP} by $P_{\eta,s}u_0$, and $P^{\rho}_{\eta,s}u_0$. $(P_{\eta,s})_{s \in [\eta,T)}$, and $(P^{\rho}_{\eta,s})_{s \in [\eta,T)}$ are called the heat semigroups associated with $k$, and $k_{\rho}$.

In the time-homogeneous case, i.e., when $k$ does not depend on $t$, the existence of $(P_{\eta,s})_{s \in [\eta,T)}$ and $(P^{\rho}_{\eta,s})_{s \in [\eta,T)}$ is guaranteed by symmetric Dirichlet form theory. The existence of the fundamental solution classically follows from so-called ultracontractivity estimates for the heat semigroup which are a consequence of Nash's inequality. For time-inhomogeneous jumping kernels $k$ which satisfy the following pointwise lower bound for some $\lambda > 0$
\begin{equation}
\label{eq:ptwlower}
k(t;x,y) \ge \lambda \vert x-y \vert^{-d-\alpha}, ~~ t \in (0,T), ~ x,y \in \R^d,
\end{equation}
the existence of the fundamental solutions $p$ and $p_{\rho}$ was proved in \cite{MaMi13} by approximation of $k$ through a sequence of smooth jumping kernels for which the desired properties follow from the theory of pseudo-differential operators. A similar result is proved in \cite{Kom88}, \cite{Kom95} but under an additional smoothness assumption on $t$.

The following result explains the connection between $p$ and $p_{\rho}$ and is crucial to our approach. It is frequently used in the derivation of upper heat kernel bounds for alpha-stable like processes and goes back to a probabilistic construction carried out in \cite{Mey75}. An analytic proof via the parabolic maximum principle is derived in  \cite{GrHu08} (see also \cite{GHL14}). Since both proofs are known only in the time-homogeneous case, we will provide a modified version of the argument in \cite{GrHu08} in the appendix.

\begin{lemma} 
\label{lemma:MeyersDec}
Assume that $k$ satisfies \eqref{eq:kupper}, \eqref{eq:klower}. Then there exists $c > 0$ such that for every $\rho > 0$, $0 \le \eta < s < T$, and $x,y \in \R^d$ it holds
\begin{align}
\label{eq:nontrunctrunc}
p(y,s;x,\eta) &\le p_{\rho}(y,s;x,\eta) + c(s-\eta) \rho^{-d-\alpha},\\
\label{eq:truncnontrunc}
p_{\rho}(y,s;x,\eta) &\le e^{c\rho^{-\alpha}(s-\eta)} p(y,s;x,\eta).
\end{align}
\end{lemma}

Next, we provide the so-called on-diagonal bound for the heat kernels $p$, and $p_{\rho}$.

\begin{theorem}[on-diagonal bound]
\label{thm:ondiag}
Assume that $k$ satisfies \eqref{eq:kupper}, \eqref{eq:klower}. Then there exists $c > 0$ depending on $d,\alpha,\lambda,\Lambda$ such that for every $\rho > 0$, $0 \le \eta < s < T$, and $x,y \in \R^d$ it holds
\begin{align}
\label{eq:pondiag}
p(y,s;x,\eta) &\le c(s-\eta)^{-\frac{d}{\alpha}},\\
\label{eq:truncondiag}
p_{\rho}(y,s;x,\eta) &\le ce^{c\rho^{-\alpha}(s-\eta)}(s-\eta)^{-\frac{d}{\alpha}}.
\end{align}
\end{theorem} 

There are at least two ways to prove \autoref{thm:ondiag}. One approach classically goes via Nash inequalities (see \cite{CKW21}) and can be traced back to Nash's famous work \cite{Nas58}. This proof also works in the time-inhomogeneous setup (see \cite{MaMi13}).\\
Another way to establish on-diagonal bounds goes via $L^{\infty}-L^1$-estimates. For this, we refer to \autoref{lemma:mixedLinftyL2}, where such estimate is proved in a more general setup. Observe that $(t,z) \mapsto p(y,t;z,\eta)$ solves $\partial_t u - L u = 0$ in $(\eta,T) \times \R^d$ for every $y \in \R^d$. Therefore, by the $L^{\infty}-L^1$-estimate \eqref{eq:mixedLinftyL1}, for every $0 \le \eta < s < T$ and $x,y \in \R^d$:
\begin{align*}
p(y,s;x,\eta) \le c (s-\eta)^{-\frac{d}{\alpha}} \sup_{t \in (\eta,s)}\int_{\R^d} p(y,t;z,\eta) \d z \le  c (s-\eta)^{-\frac{d}{\alpha}},
\end{align*}
where we used \eqref{eq:pint} in the last step. \eqref{eq:truncondiag} follows from \autoref{lemma:MeyersDec}.\\

The remainder of this section is devoted to proving an $L^{\infty}-L^2$-estimate for solutions to $\partial_t u - L^{\rho}_t u = 0$ in a time-space cylinder $I_R(t_0) \times B_R(x_0)$, where $t_0 \in (0,T)$, $x_0 \in \R^d$, and $I_R(t_0) := (t_0 - R^{\alpha},t_0) \subset (\eta,T)$.\\
Since we consider only truncated jumping kernels, it is possible to estimate the tail-term by an $L^2$-norm over a ball with radius $\rho$, without any pointwise lower bounds of $k$ or a UJS-type condition, see \cite{BBK09}. This is at the cost of a suboptimal scaling constant in the resulting estimate, which luckily does not affect the proof of the final heat kernel estimate.

A similar result for solutions to elliptic equations was obtained in \cite{CKW21}. We follow the strategy outlined in \cite{Str18} which is based on a nonlocal adaptation of De Giorgi's iteration technique (see \cite{CKP14},\cite{CKP16}) but present the proof in all details due to our special treatment of the tail term.

\begin{lemma}[truncated $L^{\infty}-L^2$-estimate]
\label{lemma:truncL2}
Assume that $k$ satisfies \eqref{eq:kupper}, \eqref{eq:klower}. There exists a constant $C > 0$ depending on $d,\alpha,\lambda,\Lambda$ such that for every $t_0 \in (0,T)$, $x_0 \in \R^d$, and $\rho, R > 0$ with $R \le \rho/2 \wedge t_0^{1/\alpha}$, and every subsolution $u$ to $\partial_t u - L_t^{\rho} u = 0$ in $I_R(t_0) \times B_R(x_0)$ it holds:
\begin{equation}
\sup_{I_{R/2}(t_0) \times B_{R/2}(x_0)} u \le C \left( \frac{\rho^{\alpha}}{R^{\alpha}} \right)^{\frac{d}{2\alpha}} R^{-\frac{d}{2}} \sup_{t \in I_{R}(t_0)} \left( \int_{B_{2\rho}(x_0)} u^2(t,x) \d x \right)^{1/2}.
\end{equation}
\end{lemma}

For the definition of a subsolution to $\partial_t u - L_t^{\rho} u = 0$ in $I_R(t_0) \times B_R(x_0)$, we refer to the appendix.

\begin{proof}

To keep notation simple, from now on we will write $I_R := I_R(t_0)$, and $B_R(x_0) : = B_R$. From the usual Caccioppoli inequality, see \cite{Str18}, we know that for every $r,R$ with $0 < r \le R \le \rho/2$ and every $l > 0$, it holds
\begin{align*}
&\sup_{t \in I_R} \int_{B_R} w_l^2(t,x) \d x + \int_{I_R} \int_{B_R}\int_{B_R}(w_l(s,x)-w_l(s,y))^2 k_{\rho}(s;x,y) \d y \d x \d s\\
&\le c_1\Bigg(\sigma(R,r) \int_{I_{R+r}} \int_{B_{R+r}} w_l^2(t,x) \d x \d t + \Vert w_l\Vert_{L^1(I_{R+r}\times B_{R+r})} \sup_{t \in I_{R+r}} \sup_{x \in B_{R + \frac{r}{2}}} \int_{B_{R+r}^{c}} w_l(t,y) k_{\rho}(t;x,y) \d y \Bigg),
\end{align*}
where $c_1 > 0$ is a constant, $w_l = (u-l)_+$, and $\sigma(r,R) = r^{-(\alpha \vee 1)}(R+r)^{(\alpha \vee 1) - \alpha}$.
Define $\kappa = 1 + \frac{\alpha}{d}$, and $A(l,R) : = \vert \lbrace (t,x) \in I_{R} \times B_R : u(t,x) > l \rbrace \vert$. Then, by \eqref{eq:klower} and the fractional Sobolev inequality:

\begin{align*}
\int_{I_{R}} \int_{B_R} w^2_l(t,x) \d x \d t &\le c_2\vert A(l,R) \vert^{\frac{1}{\kappa'}} \left( \int_{I_R} \int_{B_R} w_l^{2\kappa}(s,x)  \d x \d s \right)^{\frac{1}{\kappa}}\\
&\le c_3\vert A(l,R) \vert^{\frac{1}{\kappa'}} \left( \left(\sup_{t \in I_R}\int_{B_R} w_l^2(t,x) \d x\right)^{\kappa-1}\int_{I_R} \left( \int_{B_R} w_l^{\frac{2d}{d-\alpha}}(s,x)  \d x\right)^{\frac{d-\alpha}{d}} \d s \right)^{\frac{1}{\kappa}}\\
&\le c_4\vert A(l,R) \vert^{\frac{1}{\kappa'}} \Bigg(\sigma(R,r) \int_{I_{R+r}} \int_{B_{R+r}} w_l^2(t,x) \d x \d t\\
&\qquad\qquad\qquad + \Vert w_l\Vert_{L^1(I_{R+r}\times B_{R+r})} \sup_{t \in I_{R+r}} \sup_{x \in B_{R + \frac{r}{2}}} \int_{B_{R+r}^{c}} w_l(t,y) k_{\rho}(t;x,y) \d y \Bigg)
\end{align*}
for some $c_2,c_3,c_4 > 0$. Observe that for some constant $c_5 > 0$ by \eqref{eq:kupper}
\begin{equation*}
\sup_{t \in I_{R+r}} \sup_{x \in B_{R + \frac{r}{2}}} \int_{B_{R+r}^{c}} w_l(t,y) k_{\rho}(t;x,y) \d y \le c_5 r^{-\alpha} \left( \frac{\rho^{\alpha}}{r^{\alpha}} \right)^{\frac{d}{\alpha}}\sup_{t \in I_{R+r}} \dashint_{B_{2\rho}} \vert u(t,y) \vert \d y.
\end{equation*}

Let us now fix $R \in (0,\rho/2]$ and define sequences $l_i = M(1-2^{-i})$, for $M > 0$ to be defined later, $r_i = 2^{-i-1}R$, $R_{i+1} = R_i - r_{i+1}$, $R_0 := R$, $A_i = \int_{I_{R_i}} \int_{B_{R_i}} w_{l_i}^2(t,x) \d x \d t$. Note that by definition: $R/2 = \lim_{i \to \infty} R_i < \dots < R_2 < R_1 < R_0 = R$, and $l_i \nearrow M$, and $\sigma(r_i,R_i) \le c_6 R^{-\alpha}2^{2i}$. Then, we deduce from the two lines above:
\begin{align*}
A_i &\le c_7 \frac{1}{(l_i - l_{i-1})^{\frac{2}{\kappa'}}} \left( \sigma(R_i,r_i) + \frac{r_i^{-\alpha} \left( \frac{\rho^{\alpha}}{r_i^{\alpha}} \right)^{\frac{d}{\alpha}}\sup_{t \in I_{R_{i-1}}} \dashint_{B_{2\rho}} \vert u(t,x) \vert \d x}{l_i - l_{i-1}} \right) A_{i-1}^{1 + \frac{1}{\kappa'}}\\
&\le c_8 \frac{2^{\frac{2i}{\kappa'}}}{M^{\frac{2}{\kappa'}}} \left( \frac{2^{2i}}{R^{\alpha}} + \frac{2^{i(d+\alpha+1)} \rho^d }{M R^{d+\alpha}} \sup_{t \in I_{R}} \dashint_{B_{2\rho}} \vert u(t,x)\vert \d x \right) A_{i-1}^{1+ \frac{1}{\kappa'}}\\
&\le c_9 \frac{2^{\gamma i}}{M^{\frac{2}{\kappa'}} R^{\alpha} } \left( 1 + \frac{\left( \frac{\rho^{\alpha}}{R^{\alpha}} \right)^{\frac{d}{\alpha}} \sup_{t \in I_{R}} \dashint_{B_{2\rho}} \vert u(t,x)\vert \d x }{M} \right) A_{i-1}^{1 + \frac{1}{\kappa'}}
\end{align*}
where $c_7,c_8,c_9 > 0$, and $\gamma = 2 + \frac{2}{\kappa'} + d + \alpha > 0$. Let us now choose
\begin{equation*}
M : = \left( \frac{\rho^{\alpha}}{R^{\alpha}} \right)^{\frac{d}{\alpha}} \sup_{t \in I_{R}} \dashint_{B_{2\rho}}\vert u(t,x)\vert \d x + C^{\frac{\kappa'^2}{2}} c_{10}^{\frac{\kappa'}{2}}R^{- \frac{\alpha \kappa'}{2}} A_0^{1/2},
\end{equation*}
where $c_{10} = 2c_9$, $C = 2^{\gamma} > 1$. Consequently, it holds
\begin{equation*}
A_i \le  \frac{c_{10}}{M^{\frac{2}{\kappa'}} R^{\alpha} } C^i A_{i-1}^{1+ \frac{1}{\kappa'}}, ~~ A_0 \le C^{-\kappa'^2} \left( \frac{c_{10}}{R^{\alpha}M^{\frac{2}{\kappa'}}} \right)^{-\kappa'}.
\end{equation*}
By Lemma 7.1 in \cite{Giu03}, we obtain
\begin{align*}
\sup_{I_{R/2} \times B_{R/2}} u &\le M = \left( \frac{\rho^{\alpha}}{R^{\alpha}} \right)^{\frac{d}{\alpha}} \sup_{t \in I_{R}} \dashint_{B_{2\rho}}\vert u(t,x)\vert \d x + C^{\frac{\kappa'^2}{2}} c_{10}^{\frac{\kappa'}{2}} R^{- \frac{\alpha \kappa'}{2}} A_0^{1/2}\\
&\le c_{11}\left( \frac{\rho^{\alpha}}{R^{\alpha}} \right)^{\frac{d}{\alpha}} \sup_{t \in I_{R}} \dashint_{B_{2\rho}}\vert u(t,x)\vert \d x + c_{11} \left( \dashint_{I_{R}} \dashint_{B_R} u^2(t,x) \d x \d t \right)^{1/2}\\
&\le c_{12} \left( \frac{\rho^{\alpha}}{R^{\alpha}} \right)^{\frac{d}{2\alpha}} R^{-\frac{d}{2}} \sup_{t \in I_{R}} \left(\int_{B_{2\rho}} u^2(t,x) \d x \right)^{1/2},
\end{align*}
for some $c_{11}, c_{12} > 0$, as desired, where we used $R \le \rho/2$.
\end{proof}

\section{Nonlocal Aronson method}
\label{sec:method}

In this section we prove \autoref{thm:main}. As is standard for proofs of heat kernel bounds for nonlocal operators, we first establish bounds for the heat kernel corresponding to the truncated jumping kernel and derive the estimate for the original jumping kernel by gluing together short and long jumps with the help of \autoref{lemma:MeyersDec} in a second step.\\
The following lemma is a nonlocal version of \eqref{eq:AronsonIdea}.

\begin{lemma}
\label{lemma:Aronson}
Assume that $k$ satisfies \eqref{eq:kupper} and let $\rho > 0$, $0 \le \eta < s < T$. Let $u \in L^{\infty}((\eta,T) \times \R^d)$ be a solution to the $\rho$-truncated Cauchy problem \eqref{eq:truncCP} in $(\eta,T) \times \R^d$. Then there exists a constant $C = C(\Lambda)$ such that for every bounded function $H : [\eta,s] \times \R^d \to [0,\infty)$ satisfying
\begin{itemize}
\item $C\Gamma^{\alpha}_{\rho}(H^{1/2},H^{1/2}) \le -\partial_t H$ in $(\eta,s) \times \R^d$,
\item $H^{1/2} \in L^2((\eta,s);H^{\alpha/2}(\R^d))$,
\end{itemize}
the following estimate holds true:
\begin{equation}
\label{eq:Aronson}
\sup_{t \in (\eta,s)} \int_{\R^d} H(t,x) u^2(t,x) \d x \le \int_{\R^d} H(\eta,x) u_0^2(x) \d x.
\end{equation}
\end{lemma}

\begin{proof}
Let $R \ge 2$ and $\gamma_R \in C^{\infty}_c(\R^d)$ such that $\gamma_R \equiv 1$ in $B_{R-1}(0)$, $\gamma_R \equiv 0$ in $\R^d \setminus B_{R}(0)$, $0 \le \gamma_R \le 1$, $\vert \nabla\gamma_R\vert \le 2$. Consequently, $\Gamma_{\rho}^{\alpha}(\gamma_R,\gamma_R)$ satisfies:
\begin{equation}
\label{eq:cdcgammaR}
\Gamma_{\rho}^{\alpha}(\gamma_R,\gamma_R)(x) \le \begin{cases}
0, ~~ x \in B_{\frac{R-1}{2}}(0),\\
c_1, ~~ x \in \R^d \setminus B_{\frac{R-1}{2}}(0)\\
\end{cases}
\end{equation}
for some constant $c_1 > 0$.
We test the equation for $u$ with the test function $\phi = \gamma_R^2Hu$ and integrate in time over $(\eta,\tau)$, where $\tau \in (\eta,s)$, and obtain:
\begin{equation*}
\int_{\eta}^{\tau} \int_{\R^d} (\partial_t u) \phi \d x \d t + \int_{\eta}^{\tau} \int_{\R^d}\int_{\R^d} (u(t,x)-u(t,z))(\phi(t,x)-\phi(t,z))k_{\rho}(t;x,z) \d z \d x \d t = 0.
\end{equation*}
Note that $\phi$ is a valid test function since for every $t \in (\eta,s)$ by assumption it holds  $\gamma_R H(t)u(t) \in H^{\alpha/2}(\R^d)$. From $(\partial_t u)u = \frac{1}{2}\partial_t(u^2)$ and integration by parts:
\begin{align*}
\int_{\R^d} u^2(\tau,x) \gamma_R^2(x) H(\tau,x) \d x &+ 2\int_{\eta}^{\tau} \int_{\R^d}\int_{\R^d} (u(t,x)-u(t,z))(\phi(t,x)-\phi(t,z))k_{\rho}(t;x,z) \d z \d x \d t\\
&=\int_{\R^d} u^2_0(x) \gamma_R^2(x) H(\eta,x) \d x +  \int_{\eta}^{\tau} \int_{\R^d} u^2(t,x) \gamma_R^2(x) \partial_t H(t,x) \d x \d t.
\end{align*}
We treat the nonlocal term by making use of the following algebraic inequality:
\begin{align*}
(u_1-u_2)&(\gamma_1^2 H_1 u_1 - \gamma_2^2 H_2 u_2) \ge (\gamma_1 H^{1/2}_1 u_1 - \gamma_2 H^{1/2}_2 u_2)^2\\
&- c_2 \left((\gamma_1 - \gamma_2)^2 (H_1 + H_2)(u_1^2+u_2^2) + (H_1^{1/2}-H_2^{1/2})^2(\gamma_1^2+\gamma_2^2)(u_1^2+u_2^2) \right),
\end{align*}
where $c_2 > 0$. Its proof is based on the following two observations:
\begin{align*}
(u_1-u_2)(\gamma_1^2 H_1 u_1 - \gamma_2^2 H_2 u_2) &= (\gamma_1 H^{1/2}_1 u_1 - \gamma_2 H^{1/2}_2 u_2)^2 -u_1u_2(\gamma_1 H^{1/2}_1 - \gamma_2 H^{1/2}_2)^2,\\
u_1u_2(\gamma_1 H^{1/2}_1 - \gamma_2 H^{1/2}_2)^2 &\le c_2(u_1^2+u_2^2)\left((\gamma_1 - \gamma_2)^2 (H_1 + H_2) + (\gamma_1^2 + \gamma_2^2) (H^{1/2}_1 - H^{1/2}_2)^2 \right).
\end{align*}

Moreover, note that for $R > 4\rho$ it holds $\gamma_R^2(z) = \gamma_R^2(x) = 1$ for every $x \in B_{\frac{R-1}{2}}(0)$ and $z \in B_{\rho}(x)$. By symmetry of $k$, \eqref{eq:kupper} and the observations from above, we deduce:
\begin{align*}
\sup_{\tau \in (\eta,s)}&\int_{\R^d} u^2(\tau,x) \gamma_R^2(x) H(\tau,x) \d x \le 
\int_{\R^d} u^2_0(x) \gamma_R^2(x) H(\eta,x) \d x\\
&\quad + \int_{\eta}^{s} \int_{\R^d} u^2(t,x) \gamma_R^2(x) \left(2c_2\Lambda\Gamma^{\alpha}_{\rho}(H^{1/2},H^{1/2})(t,x) + \partial_t H(t,x)\right) \d x \d t\\
&\quad + 2c_2\Lambda \int_{\eta}^{s} \int_{\R^d}\int_{B_{\rho}(x)} u^2(t,x)\gamma_R^2(z)(H^{1/2}(t,x) - H^{1/2}(t,z))^2 \vert x-y \vert^{-d-\alpha} \d z \d x \d t\\
&\quad + 2c_2\Lambda \int_{\eta}^{s}\int_{\R^d} u^2(t,x) H(t,x) \Gamma^{\alpha}_{\rho}(\gamma_R,\gamma_R)(x) \d x \d t\\
&\quad + 2c_2\Lambda \int_{\eta}^{s}\int_{\R^d}\int_{B_{\rho}(x)} u^2(t,z) H(t,x) (\gamma_R(x)-\gamma_R(z))^2 \vert x-y \vert^{-d-\alpha} \d z \d x \d t\\
&\le \int_{\R^d} u^2_0(x) \gamma_R^2(x) H(\eta,x) \d x\\
&\quad + \int_{\eta}^{s} \int_{\R^d} u^2(t,x) \gamma_R^2(x) \left(4c_2\Lambda\Gamma^{\alpha}_{\rho}(H^{1/2},H^{1/2})(t,x) + \partial_t H(t,x)\right) \d x \d t\\
&\quad + 2c_2\Lambda \Vert u\Vert_{\infty}^2\int_{\eta}^{s} \int_{\R^d \setminus B_{\frac{R-1}{2}(0)}} \Gamma^{\alpha}_{\rho}(H^{1/2},H^{1/2})(t,x) \d x \d t\\
&\quad + 4c_2\Lambda \Vert u\Vert_{\infty}^2 \int_{\eta}^{s}\int_{\R^d} \Gamma^{\alpha}_{\rho}(\gamma_R,\gamma_R)(x) H(t,x) \d x \d t.
\end{align*}

Now, assume that $H$ satisfies the assumption with $C = 4c_2\Lambda$. Then, using also \eqref{eq:cdcgammaR}, the following holds true:
\begin{align*}
\left(4c_2\Lambda\Gamma^{\alpha}_{\rho}(H^{1/2},H^{1/2})(t,x) + \partial_t H(t,x)\right) &\le 0, ~~ t \in (\eta,s),~ x \in \R^d,\\
\int_{\eta}^s \int_{\R^d \setminus B_{\frac{R-1}{2}}(0)}\Gamma^{\alpha}_{\rho}(H^{1/2},H^{1/2})(t,x) \d x \d t &\to 0, ~~\text{as } R \to \infty,\\
\int_{\eta}^{s}\int_{\R^d} \Gamma^{\alpha}_{\rho}(\gamma_R,\gamma_R)(x) H(t,x) \d x \d t \le c_1 \int_{\eta}^{s}\int_{\R^d \setminus B_{\frac{R-1}{2}}(0)}& H(t,x) \d x \d t \to 0, ~~ \text{as } R \to \infty.
\end{align*}

\enlargethispage*{3ex}
Upon the observation that $\gamma_R \to 1$, as $R \to \infty$, it follows
\begin{equation*}
\sup_{\tau \in (\eta,s)}\int_{\R^d} u^2(\tau,x) H(\tau,x) \d x \le \int_{\R^d} u^2_0(x)  H(\eta,x) \d x,
\end{equation*}
as desired.
\end{proof}

Our next goal is to establish the following auxiliary estimate:

\begin{theorem}
\label{thm:auxtrunc}
Assume that $k$ satisfies \eqref{eq:kupper}, \eqref{eq:klower}. Let $y \in \R^d$, $\sigma,\rho > 0$, $\eta \ge 0$. Let $u_0 \in L^2(\R^d)$ be such that $u_0 \equiv 0$ in $B_{\sigma}(y)$. Assume that $u \in L^{\infty}( (\eta,T) \times \R^d)$ is a weak solution to $\partial_t u - L_t^{\rho} u = 0$ in $(\eta,T) \times \R^d$. Then there exist $\nu > 1, C > 0$ depending on $d,\alpha,\lambda,\Lambda$ such that for every $s \in (\eta,T)$ with $ s-\eta \le \frac{1}{4\nu}\rho^{\alpha}$:
\begin{equation*}
\vert u(s,y) \vert \le C (s-\eta)^{-\frac{d}{2\alpha}}2^{\frac{\sigma}{6\rho}} \left(\frac{\rho^{\alpha}}{\nu(s-\eta)}\right)^{-\frac{\sigma}{6\rho} + \frac{1}{2} + \frac{d}{2\alpha}} \Vert u_0\Vert_{L^2(\R^d)}.
\end{equation*}
\end{theorem}

The idea to prove \autoref{thm:auxtrunc} is to find a suitable function $H$ such that \autoref{lemma:Aronson} is applicable. In the following, we present a suitable such function.

Given $y \in \R^d$, $\rho > 0$, $0 \le \eta < s < T$, $\nu > 1$ with $s-\eta \le \frac{1}{4\nu}\rho^{\alpha}$, we define $H_{y,\rho,\eta,s,\nu} = H : [\eta,s] \times \R^d \to [0,\infty)$ via
\begin{equation*}
\begin{split}
H(t,x) &:= \left(\frac{\rho^{\alpha}}{\nu[2(s-\eta) - (t-\eta)]}\right)^{-1} \wedge \left(\frac{\rho^{\alpha}}{\nu[2(s-\eta) - (t-\eta)]}\right)^{-\frac{\vert x-y \vert}{3\rho}}\\
&= e^{-\log\left( \frac{\rho^{\alpha}}{\nu[2(s-\eta) - (t-\eta)]} \right) \left( \frac{\vert x-y \vert}{3\rho} \vee 1\right)}.
\end{split}
\end{equation*}

\begin{lemma}
\label{lemma:H}
For every $C > 0$, there exists $\nu = \nu(d,\alpha, C) > 1$ such that for every  $y \in \R^d$, $\rho > 0$, $0 \le \eta < s < T$ with $s-\eta \le \frac{1}{4\nu}\rho^{\alpha}$, the function $H_{y,\rho,\eta,s,\nu} = H$ defined above satisfies
\begin{align}
\label{eq:H1}
C &\Gamma^{\alpha}_{\rho}(H^{1/2},H^{1/2}) \le -\partial_t H,~~ \text{ in }(\eta,s) \times \R^d,\\
\label{eq:H2}
H^{1/2} &\in L^2((\eta,s);H^{\alpha/2}(\R^d)).
\end{align}
\end{lemma}

\begin{proof}
Let $y \in \R^d$, $\rho > 0$, $0 \le \eta < s <T$, with $s-\eta \le \frac{1}{4\nu}\rho^{\alpha}$, where $\nu > 1$ to be chosen later. Recall:
\begin{equation}
\label{eq:H}
\begin{split}
H(t,x) &:= \left(\frac{\rho^{\alpha}}{\nu[2(s-\eta) - (t-\eta)]}\right)^{-1} \wedge \left(\frac{\rho^{\alpha}}{\nu[2(s-\eta) - (t-\eta)]}\right)^{-\frac{\vert x-y \vert}{3\rho}}\\
&= e^{-\log\left( \frac{\rho^{\alpha}}{\nu[2(s-\eta) - (t-\eta)]} \right) \left( \frac{\vert x-y \vert}{3\rho} \vee 1\right)}.
\end{split}
\end{equation}
Note that by assumption, $\frac{\rho^{\alpha}}{\nu[2(s-\eta) - (t-\eta)]} > 1$ for every $t \in [\eta,s]$. Let $t \in [\eta,s]$ be fixed. We split the proof of \eqref{eq:H1} into three cases.

Case 1: $\vert x-y \vert \le 2 \rho$.\\
In this case, trivially $\Gamma^{\alpha}_{\rho}(H^{1/2},H^{1/2})(t,x) = 0$, and
\begin{equation*}
-\partial_t H(t,x) = -\partial_t \left(\frac{\nu[2(s-\eta) - (t-\eta)]}{\rho^{\alpha}} \right) = \nu \rho^{-\alpha} > 0.
\end{equation*}
Therefore, \eqref{eq:H1} holds true for any $\nu > 0$.

Case 2: $2\rho \le \vert x-y \vert \le 3\rho$.\\
In this case, $-\partial_t H(t,x) = \nu\rho^{-\alpha}$, as in Case 1.
Moreover, let us fix $x_0 \in B_{\rho}(x)$ with $\vert x_0 - y\vert = 3\rho$ such that for every $z \in \R^d \setminus B_{3\rho}(y)$ it holds $\vert x_0-z \vert \le 2 \vert x-z \vert$. This holds true, e.g. if one chooses $x_0$ as a point on $\partial B_{3\rho}(y)$ that minimizes $\dist(x,\R^d \setminus B_{3\rho}(y))$, since then by triangle inequality: $\vert x_0 - z\vert \le \vert x_0 - x\vert + \vert x - z\vert \le 2 \vert x-z \vert$. Note that $H(t,x) = H(t,x_0)$ and $B_{\rho}(x) \subset B_{2\rho}(x_0)$. 

Therefore:
\pagebreak[2]
\begin{align*}
\Gamma^{\alpha}_{\rho}(H^{1/2},H^{1/2})&(t,x) = \int_{B_{\rho}(x) \setminus B_{3\rho}(y)} \left( H^{1/2}(t,x) - H^{1/2}(t,z) \right)^2 \vert x-z\vert^{-d-\alpha} \d z \\
&\le c_1\int_{B_{2\rho}(x_0) \setminus B_{3\rho}(y)} \left( H^{1/2}(t,x_0) - H^{1/2}(t,z) \right)^2 \vert x_0-z\vert^{-d-\alpha} \d z \\
&\le \int_{B_{2\rho}(x_0) \setminus B_{3\rho}(y)} \vert \nabla H^{1/2}(t,x_0) \vert^2 \vert x_0 - z \vert^{2-d-\alpha} \d z\\
&\le c_2 \vert \nabla H^{1/2}(t,x_0) \vert^2 \rho^{2-\alpha}\\
&= c_2 \left( (6\rho)^{-1} \log\left(  \frac{\rho^{\alpha}}{\nu[2(s-\eta) - (t-\eta)]}  \right)  e^{-\log\left( \frac{\rho^{\alpha}}{\nu[2(s-\eta) - (t-\eta)]} \right) \left( \frac{\vert x_0-y \vert}{6\rho}\right)}\right)^2 \rho^{2-\alpha} \\
&= c_3 \rho^{-\alpha} \left(\log\left(  \frac{\rho^{\alpha}}{\nu[2(s-\eta) - (t-\eta)]}  \right) \left( \frac{\rho^{\alpha}}{\nu[2(s-\eta) - (t-\eta)]} \right)^{-1/2}\right)^2\\
&\le c_4 \rho^{-\alpha}
\end{align*}
for some $c_1,c_2,c_3,c_4 > 0$. We used the inequality $\log(a) \le a^{1/2}$ in the last step. Moreover, note that the estimate $\left( H^{1/2}(t,x_0) - H^{1/2}(t,z) \right)^2 \le \vert \nabla H^{1/2}(t,x_0) \vert^2 \vert x_0 - z \vert^{2}$ is correct since $\sup_{z \in \R^d \setminus B_{3\rho}(y)} \vert \nabla H^{1/2}(t,z)\vert = \vert \nabla H^{1/2}(t,x_0)\vert$ due to $x_0 \in \partial B_{3\rho}(y)$.
Therefore, \eqref{eq:H1} holds true in this case for any $\nu > c_4C$.

Case 3: $\vert x-y \vert > 3\rho$.\\
In this case,
\begin{equation*}
-\partial_t H(t,x) = \frac{\vert x-y \vert}{3\rho [2(s-\eta) - (t-\eta)]}e^{-\log\left( \frac{\rho^{\alpha}}{\nu[2(s-\eta) - (t-\eta)]} \right) \left( \frac{\vert x-y \vert}{3\rho}\right)}.
\end{equation*}
Moreover:
\begin{align*}
\Gamma^{\alpha}_{\rho}&(H^{1/2},H^{1/2})(t,x) = \int_{B_{\rho}(x)} \left( H^{1/2}(t,x) - H^{1/2}(t,z) \right)^2 \vert x-z\vert^{-d-\alpha} \d z \\
&\le \int_{B_{\rho}(x)} \left( e^{-\log\left( \frac{\rho^{\alpha}}{\nu[2(s-\eta) - (t-\eta)]} \right) \left( \frac{\vert x-y \vert}{6\rho}\right)} - e^{-\log\left( \frac{\rho^{\alpha}}{\nu[2(s-\eta) - (t-\eta)]} \right) \left( \frac{\vert z-y \vert}{6\rho}\right)} \right)^2 \vert x-z\vert^{-d-\alpha} \d z\\
&\le \int_{B_{\rho}(x)} \sup_{z \in B_{\rho}(x)} \left\vert \nabla e^{-\log\left( \frac{\rho^{\alpha}}{\nu[2(s-\eta) - (t-\eta)]} \right) \left( \frac{\vert z-y \vert}{6\rho}\right)} \right\vert^2 \vert x-z \vert^{2-d-\alpha} \d z\\
&\le c_5 \left((6\rho)^{-1} \log \left( \frac{\rho^{\alpha}}{\nu[2(s-\eta) - (t-\eta)]} \right) e^{-\log\left(\frac{\rho^{\alpha}}{\nu[2(s-\eta) - (t-\eta)]} \right) \left( \frac{\vert x-y \vert - \rho}{6\rho}\right)} \right)^2 \rho^{2-\alpha}\\
& = c_6 \left[\log \left( \frac{\rho^{\alpha}}{\nu[2(s-\eta) - (t-\eta)]} \right)\right]^2  \left( \frac{\rho^{\alpha}}{\nu[2(s-\eta) - (t-\eta)]} \right)^{1/3} e^{-\log\left( \frac{\rho^{\alpha}}{\nu[2(s-\eta) - (t-\eta)]} \right) \left( \frac{\vert x-y \vert}{3\rho}\right)} \rho^{-\alpha}\\
&\le c_7 \frac{1}{\nu [2(s-\eta) - (t-\eta)]}e^{-\log\left( \frac{\rho^{\alpha}}{\nu[2(s-\eta) - (t-\eta)]} \right) \left( \frac{\vert x-y \vert}{3\rho}\right)}\\
&\le c_7 \frac{\vert x-y \vert}{3\rho \nu [2(s-\eta) - (t-\eta)]}e^{-\log\left( \frac{\rho^{\alpha}}{\nu[2(s-\eta) - (t-\eta)]} \right) \left( \frac{\vert x-y \vert}{3\rho}\right)}
\end{align*}

for some $c_5,c_6,c_7 > 0$. In the third inequality, we used $\vert z-y \vert \ge \vert x-y \vert - \rho$, and in the second to last step we applied the estimate $\log(a) \le c a^{1/3}$. This holds with $c > 0$ independent of $a > 1$.
Therefore, by choosing $\nu > c_7C$, \eqref{eq:H1} is satisfied also in this case.\\
Together, we have proved \eqref{eq:H1}.
Finally, note that $\Gamma^{\alpha}_{\rho}(H^{1/2},H^{1/2}) \in L^1((\eta,s) \times \R^d)$ since for $\vert x-y \vert \ge 3\rho$, we computed above
\begin{equation*}
\Gamma^{\alpha}_{\rho}(H^{1/2},H^{1/2})(t,x) \le \vert x-y \vert c^{-\frac{\vert x-y \vert}{3\rho}},
\end{equation*}
where $c > 1$ is a constant that might depend on $\eta,s,\rho$. This proves \eqref{eq:H2}.
\end{proof}

Having at hand the function $H$ defined in \eqref{eq:H}, it is possible to establish \autoref{thm:auxtrunc}.

\begin{proof}[Proof of \autoref{thm:auxtrunc}]
The idea is to apply \autoref{lemma:Aronson} with $H$ as in \autoref{lemma:H}. It follows that for every $y \in \R^d$, $0 \le \eta < s < T$ with $s-\eta \le \frac{1}{4\nu}\rho^{\alpha}$:

\begin{equation*}
\sup_{\tau \in (\eta,s)}\int_{B_{2\rho}(y)} u^2(\tau,x) H(\tau,x) \d x \le \sup_{\tau \in (\eta,s)}\int_{\R^d} u^2(\tau,x) H(\tau,x) \d x  \le \int_{\R^d \setminus B_{\sigma}(y)} u^2_0(x)  H(\eta,x) \d x.
\end{equation*}
Consequently,
\begin{equation*}
\sup_{\tau \in (\eta,s)}\int_{B_{2\rho}(y)} u^2(\tau,x) \d x \le \left( \frac{\sup_{x \in \R^d \setminus B_{\sigma}(y)} H(\eta,x)}{\inf_{\tau \in (\eta,s), x \in B_{2\rho}(y)} H(\tau,x)} \right) \Vert u_0 \Vert_{L^2(\R^d)}^2.
\end{equation*}
By the truncated $L^{\infty}-L^2$-estimate \autoref{lemma:truncL2}, applied with $R = (s-\eta)^{1/\alpha}$, $t_0 = s$, $x_0 = y$:
\begin{equation*}
\begin{split}
\sup_{(\eta',s) \times B_{\frac{1}{2}(s-\eta)^{1/\alpha}}(y)} u  &\le c_1(s-\eta)^{-\frac{d}{2\alpha}} \left( \frac{\rho^{\alpha}}{s-\eta} \right)^{\frac{d}{2\alpha}} \sup_{\tau \in (\eta,s)}\left(\int_{B_{2\rho}(y)} u^2(\tau,x)  \d x \right)^{1/2}\\
&\le c_1(s-\eta)^{-\frac{d}{2\alpha}} \left( \frac{\rho^{\alpha}}{s-\eta} \right)^{\frac{d}{2\alpha}} \left( \frac{\sup_{x \in \R^d \setminus B_{\sigma}(y)} H(\eta,x)}{\inf_{\tau \in (\eta,s), x \in B_{2\rho}(y)} H(\tau,x)} \right)^{1/2} \Vert u_0 \Vert_{L^2(\R^d)}
\end{split}
\end{equation*}
for some $c_1 > 0$, where $\eta' := s - 2^{-\alpha}(s-\eta) \in (\eta,s)$.\\
Note that there exist $c_2,c_3 > 0$ such that for $x \in \R^d \setminus B_{\sigma}(y)$ it holds 
\begin{equation*}
H(\eta,x) \le c_2 \left(\frac{\rho^{\alpha}}{2\nu(s-\eta)}\right)^{-\frac{\sigma}{3\rho}}
\end{equation*}
and for $(\tau,x) \in [\eta,s] \times B_{2\rho}(y)$ we have:
\begin{equation*}
H(\tau,x)\ge e^{-\log\left( \frac{\rho^{\alpha}}{\nu(s-\eta)} \right) \left( \frac{2\rho}{6\rho} \vee 1\right)} \ge c_3 e^{-\log\left( \frac{\rho^{\alpha}}{\nu(s-\eta)} \right)} \ge c_3 \left(\frac{\rho^{\alpha}}{\nu(s-\eta)}\right)^{-1}.
\end{equation*}
This follows directly from the definition of $H$ and $s-\eta \le \frac{1}{4\nu}\rho^{\alpha}$. 
Together, we obtain 
\begin{equation*}
\vert u(s,y)\vert \le c_4(s-\eta)^{-\frac{d}{2\alpha}} 2^{\frac{\sigma}{6\rho}} \left(\frac{\rho^{\alpha}}{\nu(s-\eta)}\right)^{-\frac{\sigma}{6\rho} + \frac{1}{2} + \frac{d}{2\alpha}} \Vert u_0 \Vert_{L^2(\R^d)}
\end{equation*}
for some $c_4 > 0$, as desired.
\end{proof}

Having proved \autoref{thm:auxtrunc}, we are now in the position to establish upper off-diagonal bounds for $p_{\rho}(y,s;x,\eta)$:

\begin{theorem}
\label{thm:offdiagtrunc}
Assume that $k$ satisfies \eqref{eq:kupper}, \eqref{eq:klower}. Then there exists $c > 0$ depending on $d,\alpha,\lambda,\Lambda$ such that for every $\rho > 0$, $0 \le \eta < s < T$, and $x,y \in \R^d$ with  $s-\eta \le \frac{1}{4\nu}\rho^{\alpha}$:
\begin{equation}
\label{eq:offdiagtrunc}
p_{\rho}(y,s;x,\eta) \le c(s-\eta)^{-\frac{d}{\alpha}} 2^{\frac{\vert x-y \vert}{12\rho}} \left( \frac{\rho^{\alpha}}{\nu(s-\eta)} \right)^{-\frac{\vert x-y \vert}{12\rho} + \frac{1}{2} + \frac{d}{2\alpha}}.
\end{equation}
\end{theorem}

\begin{proof}
Note that the on-diagonal bound \eqref{eq:truncondiag} and \eqref{eq:pint} immediately imply for every $0 \le \eta < s < T$ with $s-\eta \le \frac{1}{4\nu}\rho^{\alpha}$ and $x,y \in \R^d$:
\begin{equation}
\label{eq:weakl2esttrunc}
\left(\int_{\R^d} p_{\rho}^2(z,s;x,\eta) \d z\right)^{1/2} \le c_1 (s-\eta)^{-\frac{d}{2\alpha}}
\end{equation}
for some $c_1 > 0$. On the other hand, from \autoref{thm:auxtrunc}, it follows for every $0 \le \eta < s < T$ with $s-\eta \le \frac{1}{4\nu}\rho^{\alpha}$ and $x,y \in \R^d$:
\begin{equation}
\label{eq:l2esttrunc}
\left(\int_{\R^d \setminus B_{\sigma}(y)} p_{\rho}^2(y,s;z,\eta) \d z\right)^{1/2} \le c_2 (s-\eta)^{-\frac{d}{2\alpha}}2^{\frac{\sigma}{6\rho}} \left(\frac{\rho^{\alpha}}{\nu(s-\eta)}\right)^{-\frac{\sigma}{6\rho} + \frac{1}{2} + \frac{d}{2\alpha}}
\end{equation}
for some $c_2 > 0$. To see this, one observes that $u(t,x) = \int_{\R^d \setminus B_{\sigma}(y)} p_{\rho}(x,t;z,\eta)p_{\rho}(y,s;z,\eta)\d z$ satisfies the assumptions of \autoref{thm:auxtrunc} with 
\begin{align*}
u_0(x) = p_{\rho}(y,s;x,\eta)\mathbbm{1}_{\{\vert x-y \vert > \sigma\}}(x), ~~~~ u(s,y) = \int_{\R^d \setminus B_{\sigma}(y)} p_{\rho}^2(y,s;z,\eta) \d z.
\end{align*}

To prove \eqref{eq:offdiagtrunc}, let us fix $0 \le \eta < s < T$ with $s-\eta \le \frac{1}{4\nu}\rho^{\alpha}$, and $x,y \in \R^d$. Then we define $\sigma = \frac{1}{2}\vert x-y \vert$ and compute, using \eqref{eq:psemi}:
\begin{align*}
p_{\rho}(y,s;x,\eta) &= \int_{\R^d} p_{\rho}(y,s;z,(s-\eta)/2)p_{\rho}(z,(s-\eta)/2;x,\eta) \d z\\
&= \int_{\R^d \setminus B_{\sigma}(y)} p_{\rho}(y,s;z,(s-\eta)/2)p_{\rho}(z,(s-\eta)/2;x,\eta) \d z\\
&+ \int_{B_{\sigma}(y)} p_{\rho}(y,s;z,(s-\eta)/2)p_{\rho}(z,(s-\eta)/2;x,\eta) \d z\\
& = J_1+J_2.
\end{align*}
For $J_1$, we compute, using \eqref{eq:weakl2esttrunc}, \eqref{eq:l2esttrunc}:
\begin{align*}
J_1 &\le \left(\int_{\R^d \setminus B_{\sigma}(y)} p_{\rho}^2(y,s;z,(s-\eta)/2) \d z \right)^{1/2} \left(\int_{\R^d \setminus B_{\sigma}(y)} p_{\rho}^2(z,(s-\eta)/2;x,\eta) \d z \right)^{1/2}\\
&\le c_3(s-\eta)^{-\frac{d}{\alpha}} 2^{\frac{\vert x-y \vert}{12\rho}} \left( \frac{\rho^{\alpha}}{\nu(s-\eta)} \right)^{-\frac{\vert x-y \vert}{12\rho} + \frac{1}{2} + \frac{d}{2\alpha}}
\end{align*}
for some $c_3 > 0$. For $J_2$, observe that $B_{\sigma}(y) \subset \R^d \setminus B_{\sigma}(x)$, and therefore by \eqref{eq:psymm}:
\begin{align*}
J_2 &\le \left(\int_{\R^d \setminus B_{\sigma}(x)} p_{\rho}^2(y,s;z,(s-\eta)/2) \d z \right)^{1/2} \left(\int_{\R^d \setminus B_{\sigma}(x)} p_{\rho}^2(z,(s-\eta)/2;x,\eta) \d z \right)^{1/2}\\
&\le c_4(s-\eta)^{-\frac{d}{\alpha}} 2^{\frac{\vert x-y \vert}{12\rho}} \left( \frac{\rho^{\alpha}}{\nu(s-\eta)} \right)^{-\frac{\vert x-y \vert}{12\rho} + \frac{1}{2} + \frac{d}{2\alpha}}
\end{align*}
for some $c_4 > 0$. Together, we obtain the desired result.
\end{proof}

Bounds for the heat kernel corresponding to the truncated jumping kernel $p_{\rho}$ imply bounds for $p$ with the help of the gluing lemma \autoref{lemma:MeyersDec}. The underlying argument is known among probabilists as 'Meyer's decomposition'. For an analytic proof relying on the parabolic maximum principle we refer to the appendix.

We are now ready to provide the proof of our main result \autoref{thm:main}:


\begin{proof}[Proof of \autoref{thm:main}]
Let $x,y \in \R^d$ be fixed. By \eqref{eq:pondiag} it suffices to prove that for some constants $c_0, c_1 > 0$ and $s-\eta \le c_0 \vert x-y \vert^{\alpha}$ it holds
\begin{equation*}
p(y,s;x,\eta) \le c_1\frac{s-\eta}{\vert x-y \vert^{d+\alpha}}.
\end{equation*}
By \autoref{lemma:MeyersDec}, and \eqref{eq:kupper} we know that for every $\rho > 0$, $0 \le \eta < s < T$:
\begin{equation}
\label{eq:glueinglemma}
\begin{split}
p(y,s;x,\eta) &\le p_{\rho}(y,s;x,\eta) + c_2(s-\eta)\Vert k-k\1_{\{\vert x-y \vert \le \rho\}}\Vert_{\infty}\\
&\le p_{\rho}(y,s;x,\eta) + c_3(s-\eta) \rho^{-d-\alpha}
\end{split}
\end{equation}
for some $c_2,c_3 > 0$.
We choose $\rho = \frac{\vert x-y \vert}{12}\left(\frac{d+\alpha}{\alpha} + \frac{1}{2} + \frac{d}{2\alpha}\right)^{-1}$. Then by \eqref{eq:offdiagtrunc} and \eqref{eq:glueinglemma} it holds for $s-\eta \le \frac{1}{4\nu}\rho^{\alpha}$:
\begin{equation*}
p(y,s;x,\eta) \le c_4(s-\eta)\vert x-y \vert^{-d-\alpha},
\end{equation*}
where $c_4 > 0$, as desired.
\end{proof}

\section{Extension: Jumping kernels of mixed type on metric measure spaces}
\label{sec:extensions}

In this section, we discuss a possible extension of the nonlocal Aronson method to jumping kernels of mixed type. Moreover, we work on a general doubling metric measure space.

Let $(M,d)$ be a locally compact, separable metric space, and let $\mu$ be a positive Radon measure with full support. We assume that $(M,d,\mu)$ satisfies the volume doubling property, i.e., there exists $C > 0$, $d \in \N$ such that
\begin{equation}
\label{eq:VD}\tag{VD}
\frac{\mu(B_R(x))}{\mu(B_r(x))} \le C \left( \frac{R}{r} \right)^d, ~~ x \in M, ~ 0 < r \le R.
\end{equation}
Note that as a consequence of \eqref{eq:VD}, for every $\delta > 0$ there exist $c_1,c_2 > 0$ such that for every $R > 0$, $x,y \in M$ with $d(x,y) \le \delta R$: $c_1 \mu(B_{R}(x)) \le \mu(B_{R}(y)) \le c_2 \mu(B_{R}(x))$.

Moreover, let $\phi : [0,\infty) \to [0,\infty)$ be strictly increasing with $\phi(0) = 0$, $\phi(1) = 1$ and
\begin{equation}
\label{eq:scaling}
C^{-1} \left( \frac{R}{r} \right)^{\alpha_1} \le \frac{\phi(R)}{\phi(r)} \le C \left( \frac{R}{r} \right)^{\alpha_2}, ~~ 0 < r \le R,
\end{equation}
for some constant $C > 0$ and $0 < \alpha_1 \le \alpha_2 < 2$.\\
For a detailed discussion of the setup, we refer to \cite{CKW21}.

Consider symmetric jumping kernels $k : (0,T) \times M \times M \to \R$ satisfying for some $\Lambda > 0$
\begin{equation}
\label{eq:mixedcomp}\tag{$k_{\le}$}
k(t;x,y) \le \Lambda \mu(B_{d(x,y)}(x))^{-1} \phi(d(x,y))^{-1}, ~~ x,y \in M,
\end{equation}
and assume that there is $\mathcal{F} \subset L^2(M,\mu)$ such that $(\cE_t,\mathcal{F})$ is a regular Dirichlet form on $L^2(M,\mu)$ for every $t \in (0,T)$, where for every $u,v \in \mathcal{F}$,
\begin{equation*}
\cE_t(u,v) = \int_M \int_M (u(x)-u(y))(v(x)-v(y))k(t;x,y) \d x \d y
\end{equation*}
is defined in the usual way. For simplicity, we write $\d x := \mu(\d x)$.\\
Moreover, we assume that the Faber-Krahn inequality holds true, i.e., that there exist $c, \nu > 0$ such that for all $t \in (0,T)$, $R > 0$, $x_0 \in M$, $D \subset B_R(x_0)$ and every $u \in \mathcal{F}$ with $u \equiv 0$ in $M \setminus D$:
\begin{equation}
\label{eq:FK}\tag{FK}
\cE_t(u,u) \ge c \phi(R)^{-1}\left( \frac{\mu(B_{R}(x_0))}{\mu(D)} \right)^{\nu} \Vert u \Vert_{L^2(D)}^2.
\end{equation}
Let $p(y,s;x,\eta)$ be the fundamental solution to the Cauchy problem associated with $k$.

\begin{theorem}
\label{thm:mainmixed}
Let $(M,d,\mu)$ and $\phi$ be as above, and assume \eqref{eq:VD}, \eqref{eq:scaling}. Assume that $k$ satisfies \eqref{eq:mixedcomp} and \eqref{eq:FK}. Then there exists $c > 0$ such that for every $0 \le \eta < s < T$, $x,y \in M$:
\begin{equation}
p(y,s;x,\eta) \le c \left[ \mu(B_{\phi^{-1}(s-\eta)}(x))^{-1} \wedge \frac{s-\eta}{\mu(B_{d(x,y)}(x)) \phi(d(x,y))} \right].
\end{equation}
\end{theorem}

We remark that variants of \autoref{thm:mainmixed} for time-homogeneous jumping kernels of mixed type on doubling metric measure spaces can be found in several articles, e.g., \cite{ChKu08}, \cite{CKW21}, \cite{CKKW21}.

\begin{proof}
First of all, we observe that by \eqref{eq:VD}, \eqref{eq:scaling}, \eqref{eq:mixedcomp}:
\begin{align}
\label{eq:mixedint}
\int_{B_{R}(x)} d(x,y)^{2} k(t;x,y) \d y \le c R^2 \phi(R)^{-1}, ~~~~~~\int_{M \setminus B_{R}(x)} k(t;x,y) \d y \le c \phi(R)^{-1}.
\end{align}
for every $t \in (0,T)$, $x \in M$, $R > 0$. For a proof, see \cite{CKW21}.
Given $y \in M$, $\rho > 0$, $0 \le \eta < s < T$, $\nu > 1$ with $s-\eta \le \frac{1}{4\nu}\phi(\rho)$, we define $H_{y,\rho,\eta,s,\nu} = H : [\eta,s] \times M \to [0,\infty)$ via
\begin{equation*}
H(t,x) := \left(\frac{\phi(\rho)}{\nu[2(s-\eta) - (t-\eta)]}\right)^{-1} \wedge \left(\frac{\phi(\rho)}{\nu[2(s-\eta) - (t-\eta)]}\right)^{-\frac{d(x,y)}{3\rho}}.
\end{equation*}
With the help of \eqref{eq:mixedint}, \eqref{eq:scaling} it is easy to check along the lines of \autoref{lemma:H} that $H$ satisfies the assumptions of \autoref{lemma:Aronson}, namely for every $C > 0$ there exists $\nu > 1$ such that
\begin{align*}
C \Gamma^{\phi}_{\rho}(H^{1/2},H^{1/2}) \le -\partial_t H~~ \text{ in } (\eta,s) \times M,~~~~ H^{1/2} \in L^2((\eta,s);\mathcal{F}),
\end{align*}
where 
\begin{equation*}
\Gamma^{\phi}_{\rho}(H^{1/2},H^{1/2})(t,x) = \int_{B_{\rho}(x)} (H^{1/2}(t,x) - H^{1/2}(t,z))^2 \mu(B_{d(x,z)}(x))^{-1} \phi(d(x,z))^{-1} \d z.
\end{equation*}
Next, we observe that by \eqref{eq:VD}, \eqref{eq:mixedint}, \eqref{eq:scaling}, \eqref{eq:mixedcomp}, and \eqref{eq:FK} the following $L^{\infty}-L^2$-estimate holds true for local subsolutions $u$ to $\partial_t u - L_t^{\rho} u = 0$ in $I_R(t_0) \times B_R(x_0)$, where $I_R(t_0) = (t_0 - \phi(R),t_0)$:
\begin{equation}
\label{eq:mixedLinftyL2}
\sup_{I_{R/2}(t_0) \times B_{R/2}(x_0)} u \le C \left( \frac{\mu(B_{\rho}(x_0))}{\mu(B_R(x_0))} \right)^{\frac{1}{2}} \mu(B_R(x_0))^{-\frac{1}{2}} \sup_{t \in I_{R}(t_0)} \left( \int_{B_{2\rho}(x_0)} u^2(t,x) \d x \right)^{1/2}.
\end{equation}
We give a proof of this fact later, see \autoref{lemma:mixedLinftyL2}.\\
Applying $H$ to \autoref{lemma:Aronson}, we obtain from \eqref{eq:mixedLinftyL2} and the definition of $H$:
\begin{equation}
\label{eq:auxmixed}
\vert u(s,y) \vert \le c_1 \left( \frac{\mu(B_{\rho}(y))}{\mu(B_{\phi^{-1}(s-\eta)}(y))} \right)^{\frac{1}{2}} \mu(B_{\phi^{-1}(s-\eta)}(y))^{-\frac{1}{2}} 2^{\frac{\sigma}{6\rho}} \left(\frac{\phi(\rho)}{\nu(s-\eta)}\right)^{-\frac{\sigma}{6\rho} + \frac{1}{2}} \Vert u_0 \Vert_{L^2(D)}
\end{equation}
for $c_1 > 0$, where $u,u_0$ are as in \autoref{lemma:Aronson}, and $0 \le \eta < s < T$ with $s- \eta \le \frac{1}{4\nu} \phi(\rho)$, $\sigma > 0$.\\
Moreover, the following on-diagonal estimates hold for every $0 \le \eta < s < T$, $x,y \in M$, $\rho > 0$:
\begin{align}
\label{eq:mixedondiagp}
p(y,s;x,\eta) &\le c \mu(B_{\phi^{-1}(s-\eta)}(x))^{-1},\\
\label{eq:mixedondiagptrunc}
p_{\rho}(y,s;x,\eta) &\le c e^{c(s-\eta)\phi(\rho)^{-1}}\mu(B_{\phi^{-1}(s-\eta)}(x))^{-1},
\end{align}
for some $c > 0$. These estimates are proved in Lemma 5.1 and Section 4.4 in \cite{CKW21} using a stochastic approach. A more direct proof, using only analysis tools, goes via the $L^{\infty}-L^1$-estimate \eqref{eq:mixedLinftyL1}.\\
In fact, given $0 \le \eta < s < T$, $x,y \in M$, we apply \eqref{eq:mixedLinftyL1} to $(t,z) \mapsto p(y,t;z,\eta)$, choosing $R := \phi^{-1}(s-\eta)$, $x_0 := x$, $t_0 := s$. Using \eqref{eq:pint}, we obtain
\begin{equation*}
p(y,s;x,\eta) \le c \mu(B_{\phi^{-1}(s-\eta)}(x))^{-1} \sup_{t \in I_{\phi^{-1}(s-\eta)}(s)} \int_{M} p(y,t;z,\eta) \d z \le c \mu(B_{\phi^{-1}(s-\eta)}(x))^{-1},
\end{equation*}
as desired. \eqref{eq:mixedondiagptrunc} is a direct consequence of \eqref{eq:mixedondiagp} in the light of \eqref{eq:MDhelp2} and \eqref{eq:mixedint}. Note that the proof of \eqref{eq:MDhelp2} is written for $M = \R^d$ but works in the same way in the current setup.\\
Combining \eqref{eq:auxmixed} and \eqref{eq:mixedondiagptrunc}, we derive as in \autoref{thm:offdiagtrunc} for every $0 \le \eta < s < T$ with $s- \eta \le \frac{1}{4\nu} \phi(\rho)$, $x,y \in M$:
\begin{equation*}
p_{\rho}(y,s;x,\eta)\le c_2 \left( \frac{\rho}{\phi^{-1}(s-\eta)} \right)^{\frac{d}{2}} \mu(B_{\phi^{-1}(s-\eta)}(x))^{-\frac{1}{2}}\mu(B_{\phi^{-1}(s-\eta)}(y))^{-\frac{1}{2}} 2^{\frac{d(x,y)}{12\rho}} \left(\frac{\phi(\rho)}{\nu(s-\eta)}\right)^{-\frac{d(x,y)}{12\rho} + \frac{1}{2}}
\end{equation*}
for $c_2 > 0$. Finally, we explain how to deduce off-diagonal bounds for $p$. We get from \eqref{eq:MDhelp1}:
\begin{equation}
\label{eq:mixedMD}
p(y,s;x,\eta) \le p_{\rho}(y,s;x,\eta) + \int_{\eta}^s P_{\tau}^{\rho} K_{\rho}(y) \d \tau.
\end{equation}
We choose $\rho = \frac{d(x,y)}{12}\left(\frac{d+\alpha_1}{\alpha_1} + \frac{1}{2} + \frac{d}{2\alpha_1}\right)^{-1}$ and obtain by \eqref{eq:VD}, \eqref{eq:scaling}:
\begin{equation}
\label{eq:mixedtruncoffdiag}
\begin{split}
&p_{\rho}(y,s;x,\eta) \le c_2 \left( \frac{\rho}{\phi^{-1}(s-\eta)} \right)^{\frac{d}{2}} \left[\mu(B_{\phi^{-1}(s-\eta)}(x)) \mu(B_{\phi^{-1}(s-\eta)}(y))\right]^{-\frac{1}{2}} 2^{\frac{d(x,y)}{12\rho}} \left(\frac{\phi(\rho)}{\nu(s-\eta)}\right)^{-\frac{d(x,y)}{12\rho} + \frac{1}{2}}\\
&\le c_3 \left(\frac{d(x,y)}{\phi^{-1}(s-\eta)}\right)^{\frac{d}{2}} \left[\mu(B_{\phi^{-1}(s-\eta)}(x)) \mu(B_{\phi^{-1}(s-\eta)}(y))\right]^{-\frac{1}{2}} \left( \frac{\phi(d(x,y))}{s-\eta} \right)^{-\frac{3d}{2\alpha_1} - 1} \\
&\le c_4 \left(\frac{\phi^{-1}(\phi(d(x,y)))}{\phi^{-1}(s-\eta)}\right)^{\frac{d}{2}} \left(\frac{\mu(B_{\phi^{-1}(\phi(d(x,y)))}(x)) \mu(B_{\phi^{-1}(\phi(d(x,y)))}(y))}{\mu(B_{\phi^{-1}(s-\eta)}(x))\mu(B_{\phi^{-1}(s-\eta)}(y))}\right)^{\frac{1}{2}} \frac{\left( \frac{\phi(d(x,y))}{s-\eta} \right)^{-\frac{3d}{2\alpha_1} - 1} }{\mu(B_{d(x,y)}(x))}\\
&\le c_5 \left(\frac{\phi(d(x,y))}{s-\eta}\right)^{\frac{3d}{2\alpha_1} } \left( \frac{\phi(d(x,y))}{s-\eta} \right)^{-\frac{3d}{2\alpha_1} - 1} \mu(B_{d(x,y)}(x))^{-1}\\
&= c_5\frac{s-\eta}{\mu(B_{d(x,y)}(x)) \phi(d(x,y))}
\end{split}
\end{equation}
for $c_3,c_4,c_5 > 0$. Next, we estimate $\int_{\eta}^s P_{\tau}^{\rho} K_{\rho}(y) \d \tau$. For this, we compute by \eqref{eq:mixedcomp} and \eqref{eq:VD}:
\begin{equation*}
\begin{split}
\int_{\eta}^s P_{\tau}^{\rho} K_{\rho}(y) \d \tau &= \sum_{k=1}^{\infty} \int_{\eta}^s P_{\tau}^{\rho} \left[ \mathbbm{1}_{B_{ck\rho}(y) \setminus B_{c(k-1)\rho}(y)} K_{\rho} \right] (y) \d \tau\\
&\le c_6 \phi(d(x,y))^{-1}  \sum_{k=1}^{\infty} \int_{\eta}^s P_{\tau}^{\rho} \left[ \mathbbm{1}_{B_{ck\rho}(y) \setminus B_{c(k-1)\rho}(y)} \mu(B_{\rho}(\cdot))^{-1} \right] (y) \d \tau \\
&\le c_7\mu(B_{\rho}(x))^{-1}\phi(d(x,y))^{-1}  \sum_{k=1}^{\infty} k^{d}\int_{\eta}^s P_{\tau}^{\rho} \mathbbm{1}_{B_{ck\rho}(y) \setminus B_{c(k-1)\rho}(y)} (y) \d \tau
\end{split}
\end{equation*}
for $c_6,c_7 > 0$, and $c > 3 + \frac{6d}{\alpha_1}$. Using \eqref{eq:auxmixed}, \eqref{eq:VD} and \eqref{eq:scaling}, we estimate for $\tau \in (\eta,s)$, $k \ge 2$:
\begin{align*}
k^{d} P_{\tau}^{\rho} \mathbbm{1}_{B_{ck\rho}(y) \setminus B_{c(k-1)\rho}(y)} (y) &\le c_8 k^{d} \left( \frac{\mu(B_{k\rho}(y))}{\mu(B_{\phi^{-1}(\tau-\eta)}(y))} \right)  2^{\frac{c(k-1)}{6}} \left(\frac{\phi(\rho)}{\nu(\tau-\eta)}\right)^{-\frac{c(k-1)}{6} + \frac{1}{2}}\\
&\le c_9 k^{2d} 2^{\frac{c(k-1)}{6}} \left(\frac{\phi(\rho)}{\nu(\tau-\eta)}\right)^{-\frac{c(k-1)}{6} + \frac{1}{2} + \frac{d}{\alpha_1}}\\
&\le c_{10} k^{2d} 2^{\frac{c(k-1)}{6}} 4^{-\frac{c(k-1)}{6} + \frac{1}{2} + \frac{d}{\alpha_1}}
\end{align*}
for $c_8,c_9,c_{10} > 0$, where we use $s- \eta \le \frac{1}{4\nu} \phi(\rho)$. From \eqref{eq:pint}, it follows
\begin{equation}
\label{eq:mixedMDpart}
\begin{split}
\int_{\eta}^s P_{\tau}^{\rho} K_{\rho}(y) \d \tau &\le c_{11} \mu(B_{\rho}(x))^{-1} \phi(d(x,y))^{-1} \int_{\eta}^s \left( P^{\rho}_{\tau} \mathbbm{1}_{B_{c\rho}(y)}(y) + \sum_{k=2}^{\infty} k^{2d} 2^{\frac{ck}{6}} 4^{-\frac{ck}{6}}\right) \d \tau\\
&\le c_{12} \frac{s-\eta}{\mu(B_{d(x,y)}(x)) \phi(d(x,y))}
\end{split}
\end{equation}
for $c_{11},c_{12} > 0$. Combining \eqref{eq:mixedMD}, \eqref{eq:mixedtruncoffdiag}, \eqref{eq:mixedMDpart} we obtain the desired off-diagonal estimate for $s-\eta \le \frac{1}{4\nu}\phi(\rho)$. Together with the on-diagonal estimate \eqref{eq:mixedondiagp}, we deduce the desired result.
\end{proof}

\begin{remark}
Note that the proof of \autoref{thm:mainmixed} does not require the scaling argument from \cite{ChKu08} since we are working with an on-diagonal estimate and an $L^{\infty}-L^2$-estimate that take into account the parabolic scaling of the corresponding equation, see also \cite{CKKW21}.
\end{remark}

A central ingredient in the proof of \autoref{thm:mainmixed} is the $L^{\infty}-L^2$-estimate \eqref{eq:mixedLinftyL2} for subsolutions to $\partial_t u- L_t^{\rho} u = 0$. Its proof is similar to the proof of \autoref{lemma:truncL2}. However we will provide some details since this seems to be the first time that \eqref{eq:FK} is used for nonlocal parabolic $L^{\infty}-L^2$-estimates. Moreover, we provide an $L^{\infty}-L^1$-estimate for subsolutions to $\partial_t u - L_t u = 0$ which allows us to give a direct proof of the on-diagonal upper heat kernel estimate.
In the elliptic case, $L^{\infty}-L^2$- and $L^{\infty}-L^1$-estimates are established via \eqref{eq:FK} for example in \cite{CKW21}.

\begin{lemma}
\label{lemma:mixedLinftyL2}
Let $(M,d,\mu)$ and $k$ be as in \autoref{thm:mainmixed}. Then there exists $C_1 > 0$ such that for every $t_0 \in (0,T)$, $x_0 \in M$, $\rho, R > 0$ with $R \le \rho/2 \wedge \phi^{-1}(t_0)$ and every subsolution $u$ to $\partial_t u - L_t^{\rho} u = 0$ in $I_R(t_0) \times B_R(x_0)$ it holds: 
\begin{equation}
\label{eq:mixedtruncLinftyL2}
\sup_{I_{R/2}(t_0) \times B_{R/2}(x_0)} u \le C_1 \left( \frac{\mu(B_{\rho}(x_0))}{\mu(B_R(x_0))} \right)^{\frac{1}{2}} \mu(B_R(x_0))^{-\frac{1}{2}} \sup_{t \in I_{R}(t_0)} \left( \int_{B_{2\rho}(x_0)} u^2(t,x) \d x \right)^{1/2}.
\end{equation}
Moreover, there exists $C_2 > 0$ such that for every $t_0 \in (0,T)$, $x_0 \in M$, $R \le \phi^{-1}(t_0)$ and every subsolution $u$ to $\partial_t u - L_t u = 0$ in $I_R(t_0) \times B_R(x_0)$, with $u \ge 0$ in $I_R(t_0) \times B_R(x_0)$, it holds:
\begin{equation}
\label{eq:mixedLinftyL1}
\sup_{I_{R/2}(t_0) \times B_{R/2}(x_0)} u \le C_2 \mu(B_R(x_0))^{-1} \sup_{t \in I_{R}(t_0)} \int_{M} |u(t,x)| \d x.
\end{equation}
\end{lemma}
We refer to \autoref{sec:appendix} for the definition of a subsolution. $H^{\alpha/2}(\R^d)$ should be replaced by $\mathcal{F}$.
\begin{proof}
First, we prove \eqref{eq:mixedtruncLinftyL2}. Let $l > k > 0$, $0 < r \le R \le \rho/2$, $A(l,R) : = \vert \lbrace (t,x) \in I_{R} \times B_R(x_0) : u(t,x) > l \rbrace \vert$.
Let $u$ be a subsolution to $\partial_t u - L_t^{\rho} u = 0$. First, observe that for every $t \in I_{R+r}$:
\begin{equation*}
\sup_{x \in B_R(x_0)} \int_{B_{\rho}(x) \setminus B_{R+ \frac{r}{2}}(x_0)} \vert u(t,y) \vert k(t;x,y) \d y \le c_1 \sigma'(R,r) \dashint_{B_{2\rho}(x_0)} \vert u(t,y) \vert \d y,
\end{equation*}
where $c_1 > 0$, $\sigma'(R,r) := \sup_{x \in B_R(x_0)} \phi(r)^{-1} \frac{\mu(B_{\rho}(x_0))}{\mu(B_{r}(x))}$ and we used \eqref{eq:mixedcomp}.
By Caccioppoli's inequality and \eqref{eq:mixedint}, every subsolution $u$ to $\partial_t u - L_t^{\rho} u = 0$ in $I_R \times B_R(x_0)$ satisfies:
\begin{align*}
\sup_{t \in I_R} &\left\vert A\left(l,R+\frac{r}{2}\right)\right\vert \le \frac{1}{2}(l-k)^{-2}\sup_{t \in I_R}\int_{B_{R+\frac{r}{2}}(x_0)} w_{\frac{k+l}{2}}^2(t,x) \d x\\
&\le c_2 (l-k)^{-2}\left(\sigma(R,r) + \frac{\sigma'(R,r)}{l-k} \sup_{t \in I_{R+r}} \dashint_{B_{2\rho}(x_0)} \vert u(t,x) \vert \d x \right) \int_{I_{R+r}} \int_{B_{R+r}(x_0)} w_k^2(t,x) \d x \d t,
\end{align*}
where $c_2 > 0$, $\sigma(R,r) = \phi(r)^{-1} \vee (\phi(R+r)-\phi(R))^{-1}$ and
\begin{align*}
\phi\left(R+\frac{r}{2}\right)^{-1} &\int_{I_{R}} \left( \frac{\mu( B_{R+\frac{r}{2}}(x_0) )}{A(l,R+\frac{r}{2})}\right)^{\nu} \int_{B_R(x_0)} w_l^2(t,x) \d x \d t \le c_3\int_{I_{R}} \cE_t(\tau w_l(t) , \tau w_l(t)) \d t \\
&\le c_4 \left(\sigma(R,r) + \frac{\sigma'(R,r)}{l-k} \sup_{t \in I_{R+r}} \dashint_{B_{2\rho}(x_0)} \vert u(t,x) \vert \d x \right) \int_{I_{R+r}} \int_{B_{R+r}(x_0)} w_k^2(t,x) \d x \d t,
\end{align*}
where $c_3,c_4 > 0$, $\tau \in C^{\infty}_c(\R^d)$ is an arbitrary function such that $\tau \equiv 1$ in $B_R(x_0)$, $\tau \equiv 0$ in $B_{R+\frac{r}{2}}(x_0)$, $\Vert \nabla \tau \Vert_{\infty} \le 4 r^{-1}$, and we used \eqref{eq:FK}.
By combination of the foregoing two estimates, we obtain for some $c_5 > 0$:
\begin{equation}
\label{eq:mixedCacc}
\begin{split}
\int_{I_R}&\int_{B_R(x_0)} w_l^2(t,x) \d x \d t \le c_5\left(\int_{I_{R+r}} \int_{B_{R+r}(x_0)} w_k^2(t,x) \d x \d t\right)^{1+\nu} \times\\
& \times\frac{\phi(R+\frac{r}{2})}{\mu( B_{R + \frac{r}{2}}(x_0))^{\nu}} (l-k)^{-2\nu} \left(\sigma(R,r) + \frac{\sigma'(R,r)}{l-k} \sup_{t \in I_{R+r}} \dashint_{B_{2\rho}(x_0)} \vert u(t,x) \vert \d x \right)^{1+\nu}.
\end{split}
\end{equation}

Let us now fix $R \in (0,\rho/2]$ and define sequences $l_i = M(1-2^{-i})$, for $M > 0$ to be defined later, $r_i = 2^{-i-1}R$, $R_{i+1} = R_i - r_{i+1}$, $R_0 := R$, $A_i = \int_{I_{R_i}} \int_{B_{R_i}(x_0)} w_{l_i}^2(t,x) \d x \d t$. We deduce:
\begin{align*}
A_i &\le c_6 (l_i - l_{i-1})^{-2\nu} \frac{\phi(R_i)}{\mu( B_{R_{i-1}}(x_0))^{\nu}} \left(\sigma(R_i,r_i) + \frac{\sigma'(R_i,r_i)}{l_i - l_{i-1}} \sup_{t \in I_{R_{i-1}}} \dashint_{B_{2\rho}(x_0)} \vert u(t,x) \vert \d x \right)^{1+\nu}A_{i-1}^{1+\nu}\\
&\le c_7 2^{\gamma i} M^{-2\nu} (\phi(R)\mu( B_R(x_0)))^{-\nu} \left( 1 + M^{-1}\frac{\mu( B_{\rho}(x_0))}{\mu( B_{R}(x_0))} \sup_{t \in I_{R}} \dashint_{B_{2\rho}(x_0)} \vert u(t,x) \vert \d x \right)^{1+\nu}A_{i-1}^{1+\nu}
\end{align*}
for some $\gamma,c_6,c_7 > 0$, using that $\sigma(R_i,r_i) \le c_8 2^{\gamma_1 i} \phi(R)^{-1}$, $\sigma'(R_i,r_i) \le c_9 2^{\gamma_2 i} \phi(R)^{-1} \frac{\mu( B_{\rho}(x_0))}{\mu( B_{R}(x_0))}$ for some $c_8,c_9,\gamma_1,\gamma_2 > 0$. The latter follows from the fact that for all $x \in B_{R}(x_0)$: $\frac{\mu(B_{\rho}(x_0))}{\mu(B_{r_i}(x))} = \frac{\mu(B_{\rho}(x_0))}{\mu(B_{R}(x))} \frac{\mu(B_{R}(x))}{\mu(B_{r_i}(x))} \le c_{10} 2^{\gamma_3 i}\frac{\mu(B_{\rho}(x_0))}{\mu(B_{R}(x_0))} $, where $c_{10},\gamma_3 > 0$. Let us choose $c_{11} = c_7 2^{1+\gamma}$, and
\begin{equation*}
M := \frac{\mu( B_{\rho}(x_0))}{\mu( B_{R}(x_0))} \sup_{t \in I_{R}} \dashint_{B_{2\rho}(x_0)} \vert u(t,x) \vert \d x + 2^{-\frac{\gamma}{2\nu^2}} c_{11}^{\frac{1}{2\nu}} (\phi(R)\mu(B_{R}(x_0)))^{-\frac{1}{2}} A_0^{1/2}.
\end{equation*}
Hence:
\begin{align*}
A_i \le (c_{11} M^{-2\nu} (\phi(R)\mu(B_{R}(x_0)))^{-\nu} ) 2^{\gamma i} A_{i-1}^{1+\nu}, ~~ 
A_0 \le 2^{-\frac{\gamma}{\nu^2}} (c_{11} M^{-2\nu}(\phi(R)\mu(B_{R}(x_0)))^{-\nu})^{-\frac{1}{\nu}},
\end{align*}
and we can apply Lemma 7.1 in \cite{Giu03} to deduce that for some $c_{12} > 0$:
\begin{align*}
\sup_{I_{R/2} \times B_{R/2}(x_0)} u &\le \frac{\mu( B_{\rho}(x_0))}{\mu( B_{R}(x_0))} \sup_{t \in I_{R}} \dashint_{B_{2\rho}(x_0)} \vert u(t,x) \vert \d x + c_7 \left(\dashint_{I_R} \dashint_{B_R(x_0)} u^2(t,x) \d x \d t\right)^{1/2}\\
&\le c_{12} \left( \frac{\mu( B_{\rho}(x_0))}{\mu( B_{R}(x_0))} \right)^{\frac{1}{2}} (\mu( B_{R}(x_0))^{-\frac{1}{2}} \sup_{t \in I_{R}} \left( \int_{B_{2\rho}(x_0)} u^2(t,x) \d x \right)^{1/2}.
\end{align*}
This proves \eqref{eq:mixedtruncLinftyL2}. Let us now demonstrate how to prove \eqref{eq:mixedLinftyL1}. Let $u$ be a subsolution to $\partial_t u - L_t u = 0$. First, we provide a different estimate of the tail term. For every $t \in I_{R+r}$:
\begin{equation*}
\sup_{x \in B_R(x_0)} \int_{M \setminus B_{R+ \frac{r}{2}}(x_0)} \vert u(t,y) \vert k(t;x,y) \d y \le c_{13} \widetilde{\sigma}'(R,r)\int_{M} \vert u(t,y) \vert \d y,
\end{equation*}
where $\widetilde{\sigma}'(R,r) := \sup_{x \in B_{R}(x_0)} \mu(B_r(x))^{-1} \phi(r)^{-1}$ and we applied \eqref{eq:mixedcomp}. As in \eqref{eq:mixedCacc}, we get:
\begin{align*}
\int_{I_R}&\int_{B_R(x_0)} w_l^2(t,x) \d x \d t \le c_{14}\left(\int_{I_{R+r}} \int_{B_{R+r}(x_0)} w_k^2(t,x) \d x \d t\right)^{1+\nu} \times\\
& \times\frac{\phi(R+\frac{r}{2})}{\mu( B_{R + \frac{r}{2}}(x_0))^{\nu}} (l-k)^{-2\nu} \left(\sigma(R,r) + \frac{\widetilde{\sigma}'(R,r)}{l-k} \sup_{t \in I_{R+r}} \int_{M} \vert u(t,x) \vert \d x \right)^{1+\nu}
\end{align*}
for some $c_{14} > 0$. From now on, let $\overline{R} > 0$ be fixed. Moreover, let $0 < \overline{R}/2 \le r < R \le \overline{R}$ and define sequences $l_i = M(1-2^{-i})$, for $M > 0$ to be defined later, $r_i = 2^{-i-1}(R-r)$, $R_{i+1} = R_i - r_{i+1}$, $R_0 := R$, $A_i = \int_{I_{R_i}} \int_{B_{R_i}(x_0)} w_{l_i}^2(t,x) \d x \d t$ and deduce
\begin{align*}
A_i &\le c_{16} \frac{2^{\gamma i}}{M^{2\nu}} \frac{\phi(R)}{\mu( B_R(x_0))^{\nu}} \left( \frac{R}{R-r} \frac{1}{\phi(R-r)} + \frac{1}{M} \left( \frac{R}{R-r} \right)^{d}\frac{\phi(R-r)^{-1}}{\mu( B_{R}(x_0))} \sup_{t \in I_{R}} \int_{M} \vert u(t,x) \vert \d x \right)^{1+\nu} A_{i-1}^{1+\nu} \\
&\le c_{17} \frac{2^{\gamma i}}{M^{2\nu}} \frac{\left(\frac{R}{R-r}\right)^{\alpha_2+1}}{(\phi(R-r)\mu( B_R(x_0)))^{\nu}} \left( 1 + \frac{1}{M} \frac{\left( \frac{R}{R-r} \right)^{d-1}}{\mu( B_{R}(x_0))} \sup_{t \in I_{R}} \int_{M} \vert u(t,x) \vert \d x \right)^{1+\nu}A_{i-1}^{1+\nu},
\end{align*}
for $c_{16},c_{17},\gamma > 0$, using \eqref{eq:scaling} and that by \eqref{eq:VD}: $\mu(B_{R_i}(x_0)) \ge \mu(B_{\overline{R}/2}(x_0)) \ge c_{18} \mu(B_{\overline{R}}(x_0)) \ge c_{18}\mu(B_{R}(x_0))$, $\sigma(R_i,r_i) \le c_{19}2^{\gamma_4 i}\frac{R}{R-r}\phi(R-r)^{-1}$, $\widetilde{\sigma}'(R_i,r_i) \le c_{20} 2^{\gamma_5 i} \phi(R-r)^{-1} \left( \frac{R}{R-r} \right)^d \mu(B_R(x_0))^{-1}$.
The latter follows from the fact that for all $x \in B_R(x_0)$: $\mu(B_{r_i}(x))^{-1} = \frac{\mu(B_R)(x)}{\mu(B_{r_i}(x))} \mu(B_R(x))^{-1} \le c_{21} 2^{\gamma_6 i} \left( \frac{R}{R-r} \right)^d \mu(B_R(x_0))^{-1}$. We choose $c_{22} = c_{17}2^{1+\gamma}$ and
\begin{equation*}
M := \frac{\left( \frac{R}{R-r} \right)^{d-1}}{\mu( B_{R}(x_0))} \sup_{t \in I_{R}} \int_{M} \vert u(t,x) \vert \d x + 2^{-\frac{\gamma}{2 \nu^2}} c_{22}^{\frac{1}{2\nu}} \left[ \frac{\left(\frac{R}{R-r}\right)^{\alpha_2+1}}{(\phi(R-r)\mu( B_R(x_0)))^{\nu}} \right]^{\frac{1}{2\nu}} A_0^{1/2}
\end{equation*}
and deduce by arguments analogous to those in the first part of the proof:
\begin{align*}
\sup_{I_r \times B_r(x_0)} u &\le \left( \frac{R}{R-r} \right)^{d-1}\frac{1}{\mu( B_{R}(x_0))} \sup_{t \in I_{R}} \int_{M} \vert u(t,x) \vert \d x\\
&+ c_{23} \left[\frac{\left(\frac{R}{R-r}\right)^{\alpha_2+1}}{(\phi(R-r)\mu( B_R(x_0)))^{\nu}}\right]^{\frac{1}{2\nu}} \left(\int_{I_R} \int_{B_R(x_0)} u^2(t,x) \d x \d t \right)^ {1/2}\\
&= I_1 + I_2,
\end{align*}
where $c_{23} > 0$. We further estimate
\begin{align*}
I_2 &\le \left( \frac{R}{R-r} \right)^{ \frac{\alpha_2(1 + \nu) + 1}{2\nu}} \sup_{t \in I_R}\left(\dashint_{B_R(x_0)} u^2(t,x) \d x \right)^{1/2}\\
&\le \frac{1}{2} \sup_{I_R \times B_R(x_0)} u + c_{24} \left( \frac{R}{R-r} \right)^{ \frac{\alpha_2(1 + \nu) + 1}{\nu}} \sup_{t \in I_R} \dashint_{B_R(x_0)} \vert u(t,x) \vert \d x,
\end{align*} 
where $c_{24} > 0$ and we applied \eqref{eq:scaling} and H\"older's and Young's inequality. Together, we obtain
\begin{equation*}
\sup_{I_r \times B_r(x_0)} u  \le \frac{1}{2} \sup_{I_R \times B_R(x_0)} u + c_{25}\left( \frac{R}{R-r} \right)^{\delta} \frac{1}{\mu(B_R(x_0))} \sup_{t \in I_R} \int_M \vert u(t,x) \vert \d x
\end{equation*}
for $c_{25} > 0$ and $\delta := d-1 \vee \frac{\alpha_2(1 + \nu) + 1}{\nu}$. We can apply Lemma 1.1 in \cite{GiGu82} to the estimate above and deduce that there exists $c_{26} > 0$ such that for every $0 < \overline{R}/2 \le r < R \le \overline{R}$:
\begin{equation}
\sup_{I_r \times B_r(x_0)} u \le c_{26} \left( \frac{\overline{R}}{R-r} \right)^{\delta} \frac{1}{\mu(B_{\overline{R}}(x_0))} \sup_{t \in I_{\overline{R}}} \int_M \vert u(t,x) \vert \d x.
\end{equation}
Choosing $r = \overline{R}/2, R = \overline{R}$ implies the desired result \eqref{eq:mixedLinftyL1}.
\end{proof}

\section{Appendix}
\label{sec:appendix}

The main goal of this section is to give a proof of \autoref{lemma:MeyersDec} via analysis methods. We mainly follow the strategy carried out in \cite{GrHu08} modifying some of their arguments due to the time-inhomogeneity of the jumping kernel.

First, we introduce the notion of a subsolution to $\partial_t u - L_t u = 0$ in $I \times \Omega$ for some open interval $I \subset (\eta,T)$ and some open set $\Omega \subset \R^d$.

\begin{definition}
Let $\Omega \subset \R^d$ be open and bounded. We say that a function $u \in L^2_{loc}(I;H^{\alpha/2}(\R^d))$ with $\partial_t u \in L^1_{loc}(I ; L^2_{loc}(\Omega))$ is a subsolution to
\begin{align}
\label{eq:localPDE}
\partial_t u - L_t u = 0, ~~ \text{ in } I \times \Omega,
\end{align}
if for every $\phi \in H^{\alpha/2}(\R^d)$ with $\phi \equiv 0$ in $\R^d \setminus \Omega$ and $\phi \ge 0$:
\begin{align}
\label{eq:defsubsol}
\int_{\R^d} \partial_t u(t,x) \phi(x) \d x + \cE_t(u(t),\phi) \le 0, ~~ \text{ a.e. } t \in I.
\end{align}
In this case we say that $u$ solves $\partial_t u - L_t u \le 0$ in $I \times \Omega$. $u$ is a solution to \eqref{eq:localPDE} if \eqref{eq:defsubsol} holds for any $\phi \in H^{\alpha/2}(\R^d)$ with $\phi \equiv 0$ in $\R^d \setminus \Omega$.  (Sub-)solutions to the corresponding $\rho$-truncated problem are defined accordingly, replacing $k$ by $k_{\rho}$.
\end{definition}

The main ingredient in the proof of \autoref{lemma:MeyersDec} is the following parabolic maximum principle, which is an analog of Proposition 4.11 in \cite{GrHu08}:
\begin{lemma}[parabolic maximum principle]
\label{lemma:pmp}
Let $\Omega \subset \R^d$ be open. Assume that $k$ satisfies \eqref{eq:kupper}, \eqref{eq:klower}. Assume that $u$ solves
\begin{equation}
\begin{cases}
&\partial_t u - L_t u \le 0,~  \text{in } (\eta,T) \times \Omega,\\
&u_+(t) \in H^{\alpha/2}_{\Omega}(\R^d), ~~ \forall t \in (\eta,T),\\
&u_+(t) \to 0, ~ \text{in } L^2(\Omega), ~ \text{as } t \searrow \eta.
\end{cases}
\end{equation}

Then $u \le 0$ a.e. in $(\eta,T) \times \Omega$. The same result holds for subsolutions to $\partial_t u - L_t^{\rho} u = 0$.
\end{lemma}

\begin{proof}
By assumptions \eqref{eq:kupper}, \eqref{eq:klower} and dominated convergence theorem, we have for $f \in L^2(\R^d)$
\begin{equation*}
\lambda [f]^2_{H^{\alpha/2}(\R^d)} \le \cE_t(f,f) \le \Lambda [f]^2_{H^{\alpha/2}(\R^d)}, ~~ \forall t \in (0,T). 
\end{equation*}
Thus, $(\cE_t,H^{\alpha/2}(\R^d))$ is a regular Dirichlet form for every $t \in (0,T)$. Along the lines of Lemma 4.3 in \cite{GrHu08} one can prove
\begin{equation*}
\int_{\R^d} \partial_t u(t,x) \phi(u(t,x)) \d x \le 0, ~~ \text{ a.e. } t \in (\eta,T),
\end{equation*} 
where $\phi \in C^{\infty}(\R)$ such that $\phi \equiv 0$ in $(-\infty,0]$, $\phi > 0$ on $(0,\infty)$, and $0 \le \phi' \le 1$.
From here, the remainder of the proof follows along the lines of Proposition 4.11 in \cite{GrHu08}. The proof for subsolutions to $\partial_t u - L_t^{\rho} u = 0$ is carried out via similar arguments, using that for some $c > 0$:
\begin{equation}
\label{eq:truncprop}
[f]^2_{H^{\alpha/2}(\R^d)} \le c\rho^{-\alpha}\Vert f \Vert_{L^2(\R^d)}^2 + \cE_t^{\rho}(f,f), ~~  t \in (0,T).
\end{equation}
\end{proof}

In the following, for $f \in L^2(\Omega)$, we will denote by $(s,x) \mapsto P^{\Omega}_{\eta,s}f(x)$ the solution to $\partial_t u - L_t u = 0$ in $(\eta,T) \times \Omega$ with $\Vert P^{\Omega}_{\eta,s}f - f \Vert_{L^2(\Omega)} \to 0$ as $s \searrow \eta$, and $P^{\Omega}_{\eta,s}f \equiv 0$ in $\R^d \setminus \Omega$. We define $(s,x) \mapsto P^{\Omega,\rho}_{\eta,s}f(x)$ to be solution to the corresponding $\rho$-truncated problem.\\
In order to prove \autoref{lemma:MeyersDec}, we need an approximation result for $P_{\eta,s}f$. In the time-homogeneous case, its proof is given in Lemma 4.17 in \cite{GrHu08}. However, their argument does not work in our situation due to the lack of a resolvent operator associated with $P_{\eta,s}$.

\begin{lemma}
\label{lemma:sgapprox}
Let $(\Omega_n)_{n \in \N}$ be an increasing sequence of open subsets of $\R^d$ with $\bigcup_{n \in \N} \Omega_n = \R^d$, and $\eta \ge 0$.  Assume that $k$ satisfies \eqref{eq:kupper}, \eqref{eq:klower}. Then for every $f \in L^2(\R^d)$ it holds $P_{\eta,s}^{\Omega_n} f \to P_{\eta,s} f$, $P_{\eta,s}^{\Omega_n,\rho} f \to P^{\rho}_{\eta,s} f$ pointwise a.e. for a.e. $s \in (\eta,T)$, and every $\rho > 0$. 
\end{lemma}

\begin{proof}
We denote $u_n(s) := P^{\Omega_n}_{\eta,s} f$. By the parabolic maximum principle it follows that $u_n(s) \le u_{n+1}(s) \le \cdots \le P_{\eta,s} f$. Consequently, there exists a function $u \in L^2((\eta,T)\times \R^d)$ such that $u_n(s) \nearrow u(s) \le P_{\eta,s} f$ a.e., and by dominated convergence we have $\Vert u_n - u\Vert_{L^2((\eta,T)\times \R^d)} \to 0$.\\
It remains to prove $u(s) = P_{\eta,s}f$. 
Note that by \eqref{eq:defsubsol} it is sufficient to prove that $u_n \rightharpoonup u$ in $L^2((\eta,T);H^{\alpha/2}(\R^d))$. 
Let us test the equation for $u_n$ with $\phi =  u_n$, integrate over $(\eta,T)$. Then:
\begin{equation*}
\begin{split}
\int_{\R^d} u_n^2(T,x) \d x &+ 2\int_{\eta}^T \cE_t(u_n(t),u_n(t)) \d t \le \int_{\R^d} f^2(x)\d x.
\end{split}
\end{equation*}  
By uniform boundedness of $\Vert u_n \Vert_{L^2((\eta,T)\times \R^d)}$, we conclude that also $\Vert u_n \Vert_{L^{2}((\eta,T);H^{\alpha/2}(\R^d))}$ is uniformly bounded, which implies $u_n \rightharpoonup u$ in $L^2((\eta,T);H^{\alpha/2}(\R^d))$, as desired.\\
The desired result for $P^{\Omega_n,\rho}_{\eta,s}$ follows by the same argument, using \eqref{eq:truncprop} to conclude boundedness in $L^2((\eta,T);H^{\alpha/2}(\R^d))$.
\end{proof}

\begin{proof}[Proof of \autoref{lemma:MeyersDec}]
Let $f \in L^2(\R^d)$ with $f \ge 0$, and $\Omega_n \subset \R^d$ as in \autoref{lemma:sgapprox}. We proceed as in the proof of Proposition 4.4 in \cite{GHL14} to deduce for a.e. $0 \le \eta < s < T$ and $n \in \N$:
\begin{equation}
\label{eq:MDapprox}
P^{\Omega_n}_{\eta,s} f \le P^{\Omega_n,\rho}_{\eta,s}f + \int_{\eta}^s \left[ \int_{\R^d} P^{\Omega_n,\rho}_{\eta,\tau} f(y)  K_{\rho}(y) \d y \right] \d \tau,
\end{equation}
where $K_{\rho}(y) = \sup_{z \in \R^d, t \in (\eta,T)} k(t;z,y)\mathbbm{1}_{\{ \vert z-y \vert \ge \rho \}}(y)$.
The proof of \eqref{eq:MDapprox} directly follows by application of the parabolic maximum principle \autoref{lemma:pmp} for $L_t$ to
\begin{equation*}
u(s,x) = P_{\eta,s}^{\Omega_n} f(x) - P^{\Omega_n,\rho}_{\eta,s}f(x) - \phi_n(x) \int_{\eta}^s \left[ \int_{\R^d} P^{\Omega_n,\rho}_{\eta,\tau} f(y) K_{\rho}(y) \d y \right] \d \tau ,
\end{equation*}
where $\phi_n \in C_c^{\infty}(\R^d)$ with $0 \le \phi_n \le 1$ and $\phi_n \equiv 1$ in $\Omega_n$. Taking $n \to \infty$ in \eqref{eq:MDapprox} implies
\begin{equation*}
P_{\eta,s} f \le P^{\rho}_{\eta,s}f + \int_{\eta}^s \left[ \int_{\R^d} P^{\rho}_{\eta,\tau} K_{\rho}(y) f(y)  \d y \right] \d \tau,
\end{equation*}
where we used \autoref{lemma:sgapprox} and \eqref{eq:psymm}.\\
By \eqref{eq:sgrepresentation} it is easy to deduce for a.e. $0 \le \eta < s < T$ and a.e. $x,y \in \R^d$:
\begin{equation}
\label{eq:MDhelp1}
p(y,s;x,\eta) \le p_{\rho}(y,s;x,\eta) + \int_{\eta}^s P^{\rho}_{\eta,\tau} K_{\rho}(y) \d \tau.
\end{equation}
\eqref{eq:nontrunctrunc} is now a direct consequence of \eqref{eq:pint} and \eqref{eq:kupper}. 
The proof of \eqref{eq:truncnontrunc} follows by carrying out the same arguments as in the proof of Proposition 4.4 in \cite{GHL14} with 
\begin{equation*}
u(s,x) = P^{\Omega_n,\rho}_{\eta,s}f(x) - P^{\Omega_n}_{\eta,s}f(x)e^{(s-\eta)K_{\rho}'},
\end{equation*}
where $K_{\rho}':= \sup_{(t,x) \in (0,T) \times \R^d} \int_{\R^d \setminus B_{\rho}(x)} k(t;x,y) \d y$. This time, we apply the parabolic maximum principle \autoref{lemma:pmp} for $L^{\rho}_t$ and obtain for a.e. $0 \le \eta < s < T$ and a.e. $x,y \in \R^d$:
\begin{equation}
\label{eq:MDhelp2}
p_{\rho}(y,s;x,\eta) \le p(y,s;x,\eta) e^{(s-\eta)K_{\rho}'}.
\end{equation}
Note that by H\"older regularity of weak solutions to \eqref{eq:CP}, it is possible to obtain the desired results for every $\eta,s,x,y$.
\end{proof}
\enlargethispage*{8ex}


\end{document}